\documentclass[11pt]{amsart}
\usepackage{amsthm,amsfonts,amssymb,amsmath,oldgerm}
\numberwithin{equation}{section}
\usepackage{fullpage}
\usepackage{amsmath}
\usepackage{setspace}
\usepackage{stmaryrd}
\usepackage{mathrsfs}
\usepackage{fancyhdr}
\usepackage{graphicx}
\usepackage{enumerate}
\usepackage{bm}
\usepackage{bbm}
\usepackage{dsfont}
\usepackage{multirow}
\usepackage{psfrag}
\usepackage[font=scriptsize]{caption}
\usepackage[font=scriptsize]{subcaption}
\usepackage{listings}
\usepackage{tikz}
\usepackage{empheq}
\usepackage{cases}
\usetikzlibrary{matrix} 
\usepackage[framemethod=TikZ]{mdframed}
\mdfdefinestyle{MyFrame}{backgroundcolor=gray!20!white}
\usepackage[makeroom]{cancel}

\usepackage [english]{babel}
\usepackage [autostyle, english = american]{csquotes}
\MakeOuterQuote{"}

\usepackage[top=25mm,bottom=25mm,left=20mm,right=20mm,a4paper]{geometry}

\usepackage{hyperref,bookmark}

\usepackage[normalem]{ulem}
\definecolor{violet}{rgb}{0.580,0.,0.827}

\usepackage[textwidth= \marginparwidth,textsize=tiny]{todonotes}

\hypersetup{
    colorlinks,          
    citecolor=blue,        
    filecolor=blue,      
    urlcolor=blue,           
    linkcolor=blue,
}

\newcommand{\ols}[1]{\mskip.5\thinmuskip\overline{\mskip-.5\thinmuskip {#1} \mskip-.5\thinmuskip}\mskip.5\thinmuskip} 

\newcommand\dD{\mathrm{d}}

\def\eps{\varepsilon }

\newcommand{\sgn}{\text{\rm sgn}}

\newcommand{\beq}{\begin{equation}}
\newcommand{\eeq}{\end{equation}}
\newcommand{\beqa}{\begin{eqnarray}}
\newcommand{\eeqa}{\end{eqnarray}}
\newcommand\br{\begin{remark}}
\newcommand\er{\end{remark}}
\newcommand\bp{\begin{pmatrix}}
\newcommand\ep{\end{pmatrix}}
\newcommand{\be}{\begin{equation}}
\newcommand{\ee}{\end{equation}}
\newcommand\ba{\begin{equation}\begin{aligned}}
\newcommand\ea{\end{aligned}\end{equation}}
\newcommand\ds{\displaystyle}

\newcommand{\beg}{\begin{example}}
\newcommand{\eeg}{\end{exaplem}}
\newcommand{\bpr}{\begin{proposition}}
\newcommand{\epr}{\end{proposition}}
\newcommand{\bt}{\begin{theorem}}
\newcommand{\et}{\end{theorem}}
\newcommand{\bc}{\begin{corollary}}
\newcommand{\ec}{\end{corollary}}
\newcommand{\bl}{\begin{lemma}}
\newcommand{\el}{\end{lemma}}
\newcommand{\bd}{\begin{definition}}
\newcommand{\ed}{\end{definition}}
\newcommand{\brs}{\begin{remarks}}
\newcommand{\ers}{\end{remarks}}

\newtheorem{theorem}{Theorem}[section]
\newtheorem{proposition}[theorem]{Proposition}
\newtheorem{corollary}[theorem]{Corollary}
\newtheorem{lemma}[theorem]{Lemma}
\newtheorem{remark}[theorem]{Remark}
\newtheorem{definition}[theorem]{Definition}

\newtheorem{example}[theorem]{Example}

\newcommand{\R}{{\mathbb R}}

\newcommand\bx{{\bm x}}

\newcommand\bv{{\bm v}}

\newcommand\cJ{{\mathcal J}}
\newcommand\cK{{\mathcal K}}
\newcommand\cL{{\mathcal L}}
\newcommand\cM{{\mathcal M}}

\usetikzlibrary{fit}
\usetikzlibrary{arrows}
\numberwithin{equation}{section}
\numberwithin{figure}{section}
\hypersetup{linkbordercolor=red}

\title[Derivation of the bacterial run-and-tumble kinetic model : quantitative and strong convergence results]{Derivation of the bacterial run-and-tumble kinetic model: quantitative and strong convergence results} 



\begin{document}
 
\maketitle

\centerline{\scshape Alain Blaustein\footnote{alain.blaustein@inria.fr, Inria centre at the University of Lille, 40 avenue Halley - Bât A - Park Plaza
		59650 Villeneuve d'Ascq – France.}}

\bigskip

\begin{abstract}
	During the past century, biologists and mathematicians investigated two mechanisms underlying bacteria motion: the run phase during which bacteria move in straight lines and the tumble phase in which they change their orientation. When surrounded by a chemical attractant, experiments show that bacteria increase their run time as moving up concentration gradients, leading to a biased random walk towards favorable regions. This observation raises the following question, which has drawn intense interest from both biological and mathematical communities:
	\begin{center}
			what cellular mechanisms enable bacteria to feel concentration gradients ? 
	\end{center}
	In this article, we investigate an asymptotic regime that was proposed to explain this ability thanks to internal mechanisms. More precisely, we derive the run-and-tumble kinetic equation with concentration's gradient dependent tumbling rate from a more comprehensive model, which incorporates internal cellular mechanisms. Our result improves on previous investigations, as we obtain strong convergence towards the gradient dependent kinetic model with quantitative and formally optimal convergence rates. The main ingredient consists in identifying a set of coordinates for the internal cellular dynamics in which concentration gradients arise explicitly. Then, we use relative entropy methods in order to capture quantitative measurement of the distance between the model incorporating cellular mechanisms and the one with concentration gradient dependent tumbling rate.   
\end{abstract}

\vspace{0.1cm}

\tableofcontents

\section{Introduction}
\label{sec:1}
\setcounter{equation}{0}
\setcounter{figure}{0}
\setcounter{table}{0}
Quantitative studies of chemotaxis date back to the $19^{th}$ century with the pioneering works of Engelmann and Pfeffer, followed by numerous other scientists \cite{Beyerinck1893,Sherris_Preston_Shoesmith57,Adler66}. At that time, very little was known about the underlying mechanisms driving the movement of organisms toward and away from a chemical. In 1953, C. Patlak \cite{Patlak53} established the link between macroscopic motion of sensitive organisms and biased velocity jumps at the microscopic level. He proposed an abstract framework in which he derived the Patlak-Keller-Segel model using a kinetic description of organisms. This approach was then extended
in the early  80's by W. Alt \cite{Alt80} who specifically focused on chemo-sensitive cells and simplified Patlak's arguments. He proposed a detailed run-and-tumble kinetic equation and formally derived Patlak-Keller-Segel type model in the appropriate scaling. H. Othmer \textit{et al.} \cite{Othmer_Dunbar_Alt88} simplified further the analysis in the case where cells are not affected by the presence of chemicals. They proposed the kinetic run-and-tumble model in the form that is standard today. This model prescribes the dynamics of the kinetic distribution $\bar{p}(t,\bx,\bv)$ of bacteria  at time $t\geq 0$, over the phase space $\left(\bx,\bv\right)\in\R^d\times V$, where the spatial variable $\bx$ lies in $\R^d$ whereas velocities $\bv$ lie in a smooth open subset $V$ of $\R^d$. The distribution function $\bar{p}(t,\bx,\bv)$ satisfies a Boltzmann type equation which reads
\begin{equation}
	\label{eq:p:lim}
	\ds\frac{\partial \bar{p}}{\partial t}\,+\,\bv\cdot\nabla_\bx \bar{p}
	\,=\,Q\left[\bar{p}\right]\,.
\end{equation}
On the left-hand side of \eqref{eq:p:lim}, the transport operator $\bv\cdot\nabla_\bx \bar{p}$ describes the "running" phase during which bacteria move along the straight lines generated by their velocities $\bv$. On the right-hand side of \eqref{eq:p:lim}, the operator $Q\left[\bar{p}\right]$ describes the "tumbling" phase within which bacteria modify their velocity. Numerous experiments \cite{Macnab_Koshland72,Berg_Brown72,Zigmond74,Berg75,Berg_Tedesco75,HALL_Peterson79,Block_Segall_Berg83} enabled to measure the parameters involved in the tumbling process of various types of cells, including tumbling rate and post-tumbling velocity distribution. It appears that bacteria such as \textit{E. Coli} decrease their tumbling rate proportionally to the spatio-temporal variations of the logarithmic chemo-attractant concentration \cite{Macnab_Koshland72,Berg_Brown72,Block_Segall_Berg83,Kalinin_Jiang_Tu_Wu09,Jiang_Ouyang_Tu10}, leading to a biased random walk towards more favorable regions. Thus, the effective external signal $M(t,\bx)$ may be interpreted  as the logarithm concentration $\ln{\left(S(t,\bx)\right)}$ of the chemical at time $t\geq0$ and position $\bx\in\R^d$. It influences the tumbling process through its spatio-temporal variations, or path-wise gradient, defined as
\begin{equation}\label{def:mat:der}
D_t M (t,\bx,\bv)\,=\, \partial_t M(t,\bx) + \bv\cdot \nabla_\bx M(t,\bx)\,,\quad \forall\,(t,\bx,\bv)\in\R^+\times \R^d\times V\,.
\end{equation}
The tumbling operator $Q\left[\bar{p}\right]$ then takes the following general form \cite{Alt80,Othmer_Dunbar_Alt88,Chalub_Markowich_Perthame_Schmeiser04,Hwang_Kang_Stevens05,Bournaveas_Calvez08,Bournaveas_Calvez_Gutierrez_Perthame08}:
\begin{equation}\label{def:tumbling:op:lim}
	Q\left[\bar{p}\right]\,=\,
	\int_{V}\bar{\Lambda}\left(D_t M(t,\bx,\bv'),\bv,\bv'\right)\bar{p}\left(t,\bx,\bv'\right)
	-\bar{\Lambda}\left(D_t M(t,\bx,\bv),\bv',\bv\right)\bar{p}\left(t,\bx,\bv\right)\dD \bv'
	\,,
\end{equation}
where the tumbling kernel $\bar{\Lambda}\left(D_t M,\bv,\bv'\right)$ describes the probability for a bacteria with velocity $\bv'$ to transfer to velocity $\bv$ for a given  effective signal variation $D_t M$.\\

This phenomenological model raises the natural question of understanding the internal mechanisms which enable cells to feel and respond to external spatio-temporal variations, despite the gap that separates their body size and the characteristic scale at which these variations occur.\\

Attempts to answer such a question date back to the 50's for the reaction to light stimulations of Phycomyces \cite{Delbruck_Reichardt56}. In the case of chemotaxis, R. Macnab and D. Koshland \cite{Macnab_Koshland72} experimentally confirmed that the driving cellular mechanism for sensing gradients is memory. This observation was corroborated and completed by a phenomenological model in  \cite{Block_Segall_Berg83}. In this article H. Berg \textit{et al.} explained cell memory through a simple adaptation process. Adaptation arise from the interplay between fast internal variables which rapidely adapt, and therefore may be identified to $M(t,\bx)$, and a slower internal variable $m(t)\in\R$ which relaxes to $M(t,\bx)$ after longer periods of time. In their model, the internal variable $m(t)$ of each cell satisfies the following equation
\begin{equation}\label{intern:simplified}
\frac{\dD}{\dD t} m(t) = -\frac{1}{\eps}\left(m(t)-M(t,\bx)\right)\,,
\end{equation}
where the parameter $\eps>0$ represents the time delay of adaptation. The response of the cell is then proportional to the derivative of $m(t)$. The previous model is of course an approximation of the full dynamics and by now, extensive experimental investigations led to a fine understanding of the signaling pathways from chemoreceptors to flagellar-motor: identification of proteins involved in the process \cite{Borkovich_Kaplan_Hess_Simon89,Falke_Bass_Butler_Chervitz_Danielson97,Scharf_Fahrner_Turner_Berg98,Alon_Camarena_Surette_Aguera98}, ultra-sensitivity mechanisms \cite{Cluzel_Surette_Leibler00,Sourjik_Berg02,Sourjik04} and robustness through adaptation \cite{Barkai_Leibler97,Yi_Huang_Simon_Doyle00}. Comprehensive models for  internal cell dynamics have been subsequently proposed and experimentally tested \cite{Kalinin_Jiang_Tu_Wu09,Jiang_Ouyang_Tu10}.  We also know that the slower variable $m(t)$ accounts for the methylation level of chemoreceptors whereas faster dynamics rely on the phosporylation of the protein chain involved in the signaling pathways. A more complete model would therefore require $m(t)>0$, as proposed in \cite{Perthame_Tang_vauchelet16} for example. In this article, we stick to \eqref{intern:simplified} for simplicity and because it is sufficient for our purpose. Incorporating these internal dynamics into the kinetic description of the run-and-tumble process leads to the following equation
\begin{equation}
\label{eq:p}
\ds\frac{\partial p^\eps}{\partial t}\,+\,\bv\cdot\nabla_\bx p^\eps
\,-\, \frac{1}{\eps}\,\partial_m \left[(m-M)p^\eps \right]\,-\,\eps\,\partial^2_m p^\eps\,=\,Q^\eps\left[p^\eps\right]\,,
\end{equation}
where $p^\eps(t,\bx,\bv,m)$  is the probability distribution of bacteria at time $t\geq 0$, position $\bx\in \R^d$, with velocity $\bv \in V$ and methylation level $m\in \R$. In \eqref{eq:p}, the diffusion term $\partial^2_m p^\eps$ takes into account random fluctuation in the signaling process \cite{Korobkova_Emonet_Vilar_Jose_Shimizu_Cluzel04, Clausznitzer_Endres11}. The parameter $\eps>0$ quantifies the ratio between adaptation and observation time scales. According to our previous observations and unlike $Q\left[\bar{p}\right]$, the tumbling operator $Q^\eps$ does not depend on $D_t M$, but on the time derivative of the methylation level $m(t)$ instead, that is
\begin{equation}\label{def:tumbling:op}
Q^\eps\left[p^\eps\right]\,=\,
\int_{V}
\Lambda\left(\frac{m-M}{\eps},\bv,\bv'\right)p^\eps\left(t,\bx,\bv',m\right)
-\Lambda\left(\frac{m-M}{\eps},\bv',\bv\right)p^\eps\left(t,\bx,\bv,m\right)\dD \bv'
\,.
\end{equation}
In this article, we rigorously and quantitatively prove that, over large observation periods compared to the adaptation  time scale, that is, in the $\eps\rightarrow 0$ regime, solutions to \eqref{eq:p}-\eqref{def:tumbling:op} converge towards solutions to \eqref{eq:p:lim}-\eqref{def:tumbling:op:lim}. This result constitutes a mathematical illustration of how adaptation may enable bacteria to feel gradients and thus respond to external signal.\\
Before proceeding to the core of this article, we give an overview of related mathematical investigations. The kinetic run-and-tumble model \eqref{eq:p:lim}-\eqref{def:tumbling:op:lim} has drawn lots of attention from the mathematical community as it captures biologically relevant phenomena. One of the most striking being the existence of chemotactic waves, which have been studied by J. Saragosti \textit{et al.} in  \cite{Saragosti_Calvez_Bournaveas_Buguin_Silberzan_Perthame10}, on a modified Keller-Segel type model obtained in the macroscopic limit of \eqref{eq:p:lim}-\eqref{def:tumbling:op:lim}. The existence of chemotactic waves was then extended to the full kinetic model by V. Calvez in \cite{calvez2016}. In the kinetic case, V. Calvez \textit{et al.} \cite{Calvez_Gosse_Twarogowska17} numerically observed that several waves with different velocities may coexist, highlighting how intricate the situation may be. In addition to traveling waves, N. Bournaveas and V. Calvez \cite{Bournaveas_Calvez09} proved that, for specific tumbling kernels, spherically symmetric solutions to \eqref{eq:p:lim}-\eqref{def:tumbling:op:lim} feature finite time blow-up above a certain mass threshold, when \eqref{eq:p:lim}-\eqref{def:tumbling:op:lim} is coupled with an elliptic equation on $M(t,\bx)$. In contrast, this coupled nonlinear system has been proved to be well-posed globally in time, regardless of the initial mass, under various regularity assumptions on the tumbling kernel \cite{Chalub_Markowich_Perthame_Schmeiser04,Hwang_Kang_Stevens05,Bournaveas_Calvez08,Bournaveas_Calvez_Gutierrez_Perthame08}. The macroscopic limits of the kinetic model have been extensively investigated, both in the parabolic scaling  \cite{Alt80,Hillen_Othmer00,Chalub_Markowich_Perthame_Schmeiser04}, where we recover Patlak-Keller-Segel type equations, and in the hyperbolic scaling, where we obtain a system of balance law \cite{Perthame04,Filbet_Laurencot_Perthame05,James_Vauchelet13}. 

Let us now review the literature dedicated to the model with internal variables \eqref{eq:p}-\eqref{def:tumbling:op}. First, R. Erban and H. Othmer integrated more detailed internal dynamics to the kinetic model, in the case where the tumbling kernel is a linear function of the internal state, in dimension one \cite{Erban_Othmer04} and then in higher dimensions \cite{Erban_Othmer05}. The authors also took into account external forces \cite{Erban_Othmer07}. This approach was then extended by C. Xue and H. Othmer in order to include nonlinear dependence with respect to internal variables and time dependent external signal \cite{Xue_Othmer09}. In these articles, the authors explain how to obtain macroscopic equations in parabolic and hyperbolic scaling. The longtime and macroscopic regimes of the model were also rigorously analyzed by Y. Dolak and C. Schmeiser in \cite{Dolak_Schmeiser05}, for smooths turning kernels depending only on $D_t M$. Similarly, macroscopic Keller-Segel equations were derived from a one dimensional model by G. Si \textit{et al.} in \cite{Si_Tang_Yang14}. Finally, the mechanics that relate the internal behavior and the resulting macroscopic dynamics were further analyzed by C. Xue in \cite{Xue15}. At this stage in the literature, there was no investigation regarding the relation between the kinetic model with internal variables \eqref{eq:p}-\eqref{def:tumbling:op} and its run-and-tumble counterpart \eqref{eq:p:lim}-\eqref{def:tumbling:op:lim}. B. Perthame \textit{et al.} filled this gap in \cite{Perthame_Tang_vauchelet16} and recovered \eqref{eq:p:lim}-\eqref{def:tumbling:op:lim} from \eqref{eq:p}-\eqref{def:tumbling:op} in the limit $\eps\rightarrow0$. The present article completes the former, as we prove strong convergence with explicit rates. We also relax the assumption on the tumbling kernel $\Lambda$, merely requiring integrability and moment assumptions.\\

Let us now proceed to the core of this article. First,  we provide formal insights on the asymptotic behavior of $p^\eps$ in the regime $\eps\rightarrow 0$. This will supply the reader with an overview of the upcoming analysis. We consider non-negative and normalized solutions to \eqref{eq:p}-\eqref{def:tumbling:op}, that is $p^\eps_0:=p^\eps(t=0)\in L^1\left(\R^d\times V \times \R\right)$ for all $\eps>0$ with
\[
p^\eps_0\,\geq\,0\;\;\;\text{a.e. and}\quad \int_{\R^d\times V\times \R}  p^\eps_0\left(\bx,\bv,m\right)\,\dD \bx \dD \bv \dD m\,=\,1\,.
\]
Since equation \eqref{eq:p} is conservative, the preceding property is then satisfied for all $t\geq0$. Considering the leading order in $\eps$ in \eqref{eq:p}, we find 
\[
(m-M)\,p^\eps(t,\bx,\bv,m)
\underset{\eps\rightarrow 0}{\sim}0\,,
\]
which, due to our normalization hypothesis, means that $p^\eps$ concentrates into a Dirac distribution with respect to $m$ as $\eps$ vanishes.
To desingularize the limit, we rescale $p^\eps$ as follows
\begin{equation}\label{def:q:eps}
p^\eps(t,\bx,\bv,m)
\,=\,\frac{1}{\eps}\,q^\eps\left(t,\bx,\bv,\frac{1}{\eps}\left(m-N(t,\bx,\bv)\right)\right),
\end{equation}
for some function $N$ to be chosen later. The equation on the re-scaled distribution $q^\eps$ is derived performing the following change of variable in \eqref{eq:p}
\begin{equation}\label{change:var}
	y\,\longleftarrow \frac{1}{\eps}\left(m-N(t,\bx,\bv)\right).
\end{equation}
We find that $q^\eps$ solves 
\begin{equation*}
	\ds\frac{\partial q^\eps}{\partial t}\,+\,\bv\cdot\nabla_\bx q^\eps
	\,-\, \frac{1}{\eps}\,\partial_y \left[\left(y+D_t N +\frac{N-M}{\eps}\right)q^\eps+\partial_y q^\eps \right]\,=\,\Tilde{Q}^\eps\left[q^\eps\right]\,,
\end{equation*}
where $D_t$ is defined in \eqref{def:mat:der}.
To desingularize the limit as $\eps\rightarrow 0$, we choose $N$ as the solution to 
\begin{equation}\label{eq:N}
	\left\{
	\begin{array}{l}
		\displaystyle D_t N = \frac{1}{\eps}\left(M-N\right)
		,\\[1.5em]
		\displaystyle
		N(0,\bx,\bv)\,=\, M(0,\bx)
		\,,
	\end{array}
	\right.
\end{equation}
and therefore deduce that $q^\eps$ solves
\begin{equation}
	\label{eq:q:eps}
	\ds\frac{\partial q^\eps}{\partial t}\,+\,\bv\cdot\nabla_\bx q^\eps
	\,-\, \frac{1}{\eps}\,\partial_y \left[\,y \,q^\eps+\partial_y q^\eps \right]\,=\,\Tilde{Q}^\eps\left[q^\eps\right]\,,
\end{equation}
where the re-scaled tumbling operator $\Tilde{Q}^\eps$ is given by
\begin{align*}
	\Tilde{Q}^\eps\left[q^\eps\right]\,=\,
	\int_{V}
	\Lambda\left(y'-D_t N(t,\bx,\bv'),\bv,\bv'\right)&q^\eps\left(t,\bx,\bv',y'\right)\\[0.8em]
	&-\Lambda\left(y-D_t N(t,\bx,\bv),\bv',\bv\right)q^\eps\left(t,\bx,\bv,y\right)\dD \bv'
	\,,
\end{align*}
and where we used the shorthand notation 
\[y' = y + \left(N(t,\bx,\bv)-N(t,\bx,\bv')\right)/\eps.\]
Thanks to the change of variable \eqref{change:var}, the limit of $q^\eps$ is no longer singular. Indeed, the leading term in \eqref{eq:q:eps} as $\eps\rightarrow 0$ is given by the Fokker-Planck operator $\partial_y \left[\,y \,q^\eps+\partial_y q^\eps \right]$ whose kernel is generated by the Maxwellian distribution defined as
\begin{equation}\label{Maxwellian}
	\cM(y)\,=\,\frac{1}{\sqrt{2\pi}}\exp{\left(-\frac{1}{2}\left|y\right|^2\right)}\,,\quad \forall\,y\in\R\,.
\end{equation}
Hence, in the regime of fast adaptation, we expect 
\[
q^\eps(t,\bx,\bv,y)\underset{\eps\rightarrow 0}{\sim}\bar{q}^\eps(t,\bx,\bv)\cM(y)\,,
\]
where  $\bar{q}^\eps$ denotes the marginal of $q^\eps$ with respect to $y$, that is
\begin{equation}\label{def:q:eps:bar}
	\bar{q}^\eps(t,\bx,\bv)\,=\, 
	\int_{\R} q^\eps(t,\bx,\bv,y)\dD y\,.
\end{equation}
To determine the limiting dynamics of $\bar{q}^\eps$, we formally replace $q^\eps$ with $\bar{q}^\eps\cM$ in \eqref{eq:q:eps} and close the equation by integrating with respect to $y$. We find that $\bar{q}^\eps$ converges to a solution of \eqref{eq:p:lim}-\eqref{def:tumbling:op:lim} with tumbling kernel 
\begin{equation}\label{def:Lamda:bar}
\bar{\Lambda}(m,\bv,\bv')\,=\,
	\int_{\R}
\Lambda\left(y-m,\bv,\bv'\right)\cM(y)\dD y\,.
\end{equation}
We emphasize that the previous identity explicitly relates the tumbling kernel in the model with internal dynamics \eqref{eq:p}-\eqref{def:tumbling:op} and the resulting kernel in the run-and-tumble model \eqref{eq:p:lim}-\eqref{def:tumbling:op:lim}. These formal considerations lead to \[q^\eps(t,\bx,\bv,y)\;\underset{\eps\rightarrow0}{\sim}\;\bar{p}(t,\bx,\bv)\cM(y)\,,\]
and then, we invert the change of variable \eqref{change:var} to deduce the following formal convergence for $p^\eps$
\[
p^\eps(t,\bx,\bv,m)\;\underset{\eps\rightarrow0}{\sim}\;\bar{p}\,\mathcal{M}_{\eps,N}(t,\bx,\bv,m)\,,
\]
where $\bar{p}$ solves \eqref{eq:p:lim}-\eqref{def:tumbling:op:lim} with tumbling kernel \eqref{def:Lamda:bar} and where
$\mathcal{M}_{\eps,N}$ is defined for all $(t,\bx,\bv,m)\in\R^+\times \R^d\times V\times \R$ as
\begin{equation}\label{def:M:eps}
	\mathcal{M}_{\eps,N}(t,\bx,\bv,m)\,=\,\frac{1}{\eps}\, \mathcal{M}\left(\frac{m-N(t,\bx,\bv)}{\eps}\right)\,,
\end{equation}
with $N$ given in \eqref{eq:N}. In Theorem \ref{TH:MAIN} below, we rigorously prove this result. The asymptotic analysis of \eqref{eq:p}-\eqref{def:tumbling:op} was first carried in \cite{Perthame_Tang_vauchelet16}, in which B. Perthame \textit{et al.} obtained the weak convergence of $p^\eps$ in probability spaces. In this article, we complete their result by deriving strong convergence in $L^1\left(\R^d\times V\times \R\right)$ and by providing explicit convergence rates.
\section{Assumptions and main result}
\label{sec:result}

Before stating our main result, we specify our setting and the assumptions on the data of the problem. 

Let us first point out that the existence theory has been extensively investigated in the case where \eqref{eq:p:lim}-\eqref{def:tumbling:op:lim} is coupled with an equation on the effective signal $M(t,\bx)$ \cite{Chalub_Markowich_Perthame_Schmeiser04,Hwang_Kang_Stevens05,Bournaveas_Calvez08,Bournaveas_Calvez_Gutierrez_Perthame08}. In our case, the external signal is given, which makes equations \eqref{eq:p:lim}-\eqref{def:tumbling:op:lim}  and \eqref{eq:p}-\eqref{def:tumbling:op} linear. Therefore, existence and uniqueness of solutions is well established. Furthermore, $L^\infty$ \textit{a priori} bounds have been derived in \cite{Perthame_Tang_vauchelet16} for  \eqref{eq:p:lim}-\eqref{def:tumbling:op:lim}. In what follows, we will not discuss well-posedeness issues and consider strong solutions to both \eqref{eq:p:lim}-\eqref{def:tumbling:op:lim}  and \eqref{eq:p}-\eqref{def:tumbling:op} in order to focus on our main concern, that is, the regime $\eps\rightarrow 0$.

We now specify our assumptions on the data of the problem. First, we suppose that the effective signal $M$ is a given function verifying
\begin{equation}\label{hyp:M}
	\left(M:(t,\bx)\longmapsto M(t,\bx)\right)\in W^{2,\infty}\left(\R^+\times\R^d\right),
\end{equation}
and we use the notation $M_0:=M(t=0)$. Second, we suppose that the tumbling kernel $\Lambda$ takes positive values and satisfies the following general assumptions 
\begin{subequations}
	\begin{numcases}{}
		\label{hyp:Lambda+}
		\ds
		\sup_{(\bv,m)\in V\times \R}\left(
		\int_{V}
		\left(1+|\bv'|^2\right)
		\Lambda\left(m,\bv',\bv\right) \dD \bv'\right)<\,+\infty\,,\\[0,9em]
		\label{hyp:Lambda-}
		\ds
		\sup_{(\bv,m)\in V\times \R}\left(
		\int_{V}
		\Lambda\left(m,\bv,\bv'\right) \dD \bv'\right)<\,+\infty\,.
	\end{numcases}
\end{subequations}
Regarding the sequence of initial conditions $\left(p^\eps_0\right)_{\eps>0}$, we assume that it satisfies the following moment and entropy conditions
\begin{equation}\label{hyp1:p:eps}
	\sup_{\eps>0}\left(
	\int_{\R^d\times V\times \R}
	\left( \left|\log{\left(p^\eps_0(\bx,\bv,m)\right)}\right|\,+\,|\bx|+
	|\bv|^2+|m|^2\right)p^\eps_0(\bx,\bv,m)\,
	\dD \bx\dD \bv \dD m\right)<\,+\infty\,.
\end{equation}
In addition, we suppose the initial data to be well prepared in the following sense
\begin{equation}\label{hyp2:p:eps}
	\sup_{\eps>0}\left(\eps^{-2}
	\int_{\R^d\times V\times \R}
	|m-M_0(\bx)|^2\,p^\eps_0(\bx,\bv,m)\,
	\dD \bx\dD \bv \dD m\right)<\,+\infty\,.
\end{equation}
This assumption will be thoroughly discussed below our main result Theorem \ref{TH:MAIN}. To conclude with the assumption on $\left(p^\eps_0\right)_{\eps>0}$, we assume that the marginals with respect to $m$
converge to the initial data $\bar{p}_0$ of the solution $\bar{p}$ to \eqref{eq:p:lim}-\eqref{def:tumbling:op:lim}, that is
\begin{equation}\label{hyp0:p}
\sup_{\eps>0}\left(
\eps^{-2}
\left\|\bar{p}_0-\bar{p}^\eps_0\right\|_{L^1\left(\R^d\times V\right)}\right)\,<\,+\infty\,\quad \text{where}\quad
	\bar{p}^\eps_0(\bx,\bv)\,=\, 
\int_{\R} p^\eps_0(\bx,\bv,m)\dD m\,.
\end{equation}
Regarding the initial data of the limiting model, we suppose that it satisfies
\begin{equation}\label{hyp1:p}
	\int_{\R^d\times V}
	|\bv|^2\,\bar{p}_0(\bx,\bv)\,
	\dD \bx\dD \bv<\,+\infty\,.
\end{equation}
Our main result establishes a rigorous and quantitative link between equations \eqref{eq:p:lim}-\eqref{def:tumbling:op:lim} and \eqref{eq:p}-\eqref{def:tumbling:op}, providing a mathematical description of the mechanisms allowing bacteria to feel and respond to environmental gradients. More precisely, we prove that the marginal $\bar{p}^\eps$ of $p^\eps$  converges towards the solution $\bar{p}$ to \eqref{eq:p:lim}-\eqref{def:tumbling:op:lim} in $L^1\left(\R^d\times V\times \R\right)$ with pointwise and polynomial control in time. Furthermore, we characterize the blow-up profile in $m$ by proving that the methylation level distribution converges in $L^1$ towards a Maxwellian distribution with vanishing variance and whose center is explicitly determined. Our estimates are quantitative since we obtain explicit convergence rates with respect to the scaling parameter $\eps$. To conclude, the time dependence is also explicit: we achieve global in time estimates with polynomial control, avoiding exponential blow-up. The first part of our result features linear growth, ensuring that our estimate stays valid over large time periods
\begin{theorem}\label{TH:MAIN} 
	Under assumptions \eqref{hyp:M} on $M$, \eqref{hyp:Lambda+}-\eqref{hyp:Lambda-} on $\Lambda$,  \eqref{hyp1:p:eps}-\eqref{hyp0:p} on $\left(p^\eps_0\right)_{\eps>0}$ and \eqref{hyp1:p} on $\bar{p}_0$, consider a sequence of strong solutions $\left(p^\eps\right)_{\eps>0}$ to \eqref{eq:p}-\eqref{def:tumbling:op} with initial data $\left(p^\eps_0\right)_{\eps>0}$ and a strong solution $\bar{p}$ to \eqref{eq:p:lim}-\eqref{def:tumbling:op:lim} with initial data $\bar{p}_0$ and tumbling kernel $\bar{\Lambda}$ given in \eqref{def:Lamda:bar}. \\
	There exist two positive constants $\eps_0$ and $C$ such that for all $\eps$ between $0$ and $\eps_0$ and all $t\geq0$, it holds
	\[
	\left\|\,
	p^\eps\,-\,\ols{p}^\eps\,
	\mathcal{M}_{\eps,N}\,
	\right\|_{L^2\left([\,0\,,\,t\,]\,,\,L^1
		\left(
		\R^d\times V\times \R
		\right)\right)
	}\,\leq\,
	C\,\sqrt{\eps}\,\left(t \,+\, 1 \,+\, \left|\log{\left(\eps\right)}\right|^{\frac{1}{2}}\right)\,,
	\]
	and
	\[
	\left\|\bar{p}^\eps(t)-\bar{p}(t)\right\|_{L^1
		\left(
		\R^d\times V\times \R
		\right)}
	\,\leq\,
	C\left(
	\sqrt{\eps\, t}\,\left(t \,+\, 1\,+\, \left|\log{\left(\eps\right)}\right|^{\frac{1}{2}}\right)
	+
	\eps^2\right)
	\,,
	\]
where $\mathcal{M}_{\eps,N}$ is defined in \eqref{def:M:eps} and where the solution $N$ to \eqref{eq:N} satisfies 
	\begin{equation*}
		\left|D_t M(t,\bx)
		-D_t N(t,\bx,\bv)\right|
		\,\leq\,C\left(\eps+e^{-\frac{t}{\eps}}\right)\left(1+|\bv|^2\right) \,,\quad \forall(\bx,\bv)\in \R^d\times V\,,
	\end{equation*}
	with $D_t$ is defined in \eqref{def:mat:der}. In these estimates, the constant $C>0$ can be computed explicitly and only depends on the constants in assumptions \eqref{hyp:M}-\eqref{hyp1:p}.
\end{theorem}
The main difficulty to obtain Theorem \ref{TH:MAIN} consists in proving that the rescaled distribution $q^\eps$ introduced in \eqref{eq:q:eps} converges towards $\bar{p}\,\cM$. This proof is divided into two main steps, corresponding to the formal computations developed in the introduction. First, we prove that the entropy dissipation along the flow of \eqref{eq:q:eps} controls the $L^1$ convergence of $q^\eps$ towards its Maxwellian local equilibrium $\bar{q}^\eps\,\cM$. Then, we deduce from the first step the convergence of $\bar{q}^\eps$ towards the solution $\bar{p}$ to \eqref{eq:p:lim}-\eqref{def:tumbling:op:lim}. This step relies on a fine decomposition of the error between $\bar{q}^\eps$ and $\bar{p}$. More precisely, we take advantage of the dissipative structure of equation \eqref{eq:p:lim} by identifying signed contributions in the time fluctuations of the $L^1$ distance that separates $\bar{q}^\eps$ and $\bar{p}$.\\

Before proceeding to the proof, let us comment on Theorem \ref{TH:MAIN}. Its main feature lies in the explicit convergence rates with respect to the adaptation/observation ratio $\eps$, completed by the strong convergence in $L^1$. On the biological point of view, this may give insights regarding the range of validity for the model \eqref{eq:p:lim}-\eqref{def:tumbling:op:lim}. Furthermore, pointwise convergence in time  is achieved for $\bar{p}^\eps$ and our estimates are global in time with at most polynomial growth. This also provides insights regarding the time scales over which the internal behavior of bacteria may have measurable contributions that the kinetic run-and-tumble model \eqref{eq:p:lim}-\eqref{def:tumbling:op:lim} cannot account for. We also emphasize that very mild assumptions are taken on the tumbling kernel $\Lambda$ compared to what is usually done in the literature. To conclude, we point out that the constants in our estimates are explicitly computable  in terms of the bounds assumed from assumption \eqref{hyp:M} through \eqref{hyp1:p}. In particular, we do not require regularity on the solution $p^\eps$ to carry out our proof.\\
The main restriction of Theorem \ref{TH:MAIN} lies in the well-preparedness assumption \eqref{hyp2:p:eps}. It may be removed at the cost of additional technicalities. To do so, the main ingredient lies in considering the following time dependent scaling instead of \eqref{change:var} 
\[
	y\,\longleftarrow \frac{1}{\theta^\eps(t)}\left(m-N(t,\bx,\bv)\right),
\]
where $\theta^\eps$ is explicitly given by $\theta^\eps(t)  = \eps+ (1-\eps)e^{-2t/\eps}$. The subsequent equation on $q^\eps$ remains unchanged except for we substitute $\eps$ by $\theta^\eps(t)$. We choose not to follow this path here but provide additional details in the conclusion at the end of this article.
\section{A priori estimates}
\label{sec:a:priori:estimate}
In this section, we derive uniform in $\eps$ moment estimates for the solution $q^\eps$ to \eqref{eq:q:eps} and prove that the  solution $N$ to \eqref{eq:N} converges to the effective signal $M$.\\

The following Lemma ensures that the solution $N$ to \eqref{eq:N} is globally Lipschitz with respect to $\bv$, uniformly with respect to $\eps$ and also that $N$ converges to $M$ as $\eps$ vanishes. The proof relies on an explicit formula for $N$ in terms of the given external signal $M$.
\begin{lemma}\label{estimate:N} Under assumption \eqref{hyp:M}, the solution $N$ to \eqref{eq:N} satisfies 
	\begin{equation}\label{estimate:N:1}
		\left|N(t,\bx,\bv)-N(t,\bx,\bv')\right|
		\leq \left\|\nabla_\bx M\right\|_{L^\infty}\eps \left|\bv-\bv'\right| ,
	\end{equation}
	and
	\begin{equation}\label{estimate:N:2}
	\left|D_t M(t,\bx)
	-D_t N(t,\bx,\bv)\right|
	\,\leq\,4\left\| M\right\|_{W^{2,\infty}}\left(\eps+e^{-\frac{t}{\eps}}\right)\left(1+|\bv|^2\right) \,,
	\end{equation}
	for all $(t,\bx,\bv,\bv')\in \R^+\times\R^d\times V^2$, with $D_t$ is defined in \eqref{def:mat:der}.
\end{lemma}
\begin{proof} We consider $(t,\bx,\bv,\bv')\in \R^+\times\R^d\times V^2$. First, we point out that since $N$ solves \eqref{eq:N}, it is given by the following explicit formula
	\begin{equation}\label{formula:N}
	N(t,\bx,\bv)
	=
	M(0,\bx-t\bv)\, e^{-\frac{t}{\eps}}
	+
	\frac{1}{\eps}
	\int_0^t
	M(s,\bx+(s-t)\bv)\,e^{\frac{s-t}{\eps}}\dD s\,.
	\end{equation}
	
	Let us first derive \eqref{estimate:N:1}. According to \eqref{formula:N}, we have
	\begin{align*}
		N(t,\bx,\bv)-N(t,\bx,\bv')
		=
		&\left(M(0,\bx-t\bv)-M(0,\bx-t\bv')\right)e^{-\frac{t}{\eps}}\\[0.8em]
		&+
		\frac{1}{\eps}
		\int_0^t
		\left(M(s,\bx+(s-t)\bv)-M(s,\bx+(s-t)\bv')\right)e^{\frac{s-t}{\eps}}\dD s \,.
	\end{align*}
	Taking the absolute value on both sides in the preceding relation and applying the triangular inequality, this yields:
	\begin{align*}
		\left|N(t,\bx,\bv)-N(t,\bx,\bv')\right|
		\leq
		&\left|M(0,\bx-t\bv)-M(0,\bx-t\bv')\right|e^{-\frac{t}{\eps}}\\[0.8em]
		&+
		\frac{1}{\eps}
		\int_0^t
		\left|M(s,\bx+(s-t)\bv)-M(s,\bx+(s-t)\bv')\right|e^{\frac{s-t}{\eps}}\dD s\,.
	\end{align*}
	According to \eqref{hyp:M}, $M$ is uniformly Lipschitz. Hence, we deduce
	\begin{equation*}
		\left|N(t,\bx,\bv)-N(t,\bx,\bv')\right|
		\leq C \left|\bv-\bv'\right| t \,e^{-\frac{t}{\eps}}
		+
		\frac{C}{\eps}\left|\bv-\bv'\right|
		\int_0^t
		(t-s) \,e^{\frac{s-t}{\eps}}
		\dD s\,,
	\end{equation*}
	where $C=\left\|\nabla_\bx M\right\|_{L^\infty}$. In the previous right-hand side, we estimate $t\, e^{-t/\eps}$ with $\eps$ and the time integral  thanks to the following relation
	\[
	\int_0^t
	(t-s) e^{\frac{s-t}{\eps}}\,
	\dD s\,=\,\int_0^t
	\sigma\, e^{-\frac{\sigma}{\eps}}\,
	\dD \sigma
	\,=\, \eps^2 - \left(\eps \,t\, e^{-\frac{t}{\eps}} + \eps^2 e^{-\frac{t}{\eps}}\right)\leq \eps^2,
	\]
	where the first equality is obtained through the change of variable $\sigma \leftarrow t-s$ whereas the last equality is obtained integrating by part with respect to $\sigma$. It yields the expected result
	\begin{equation*}
		\left|N(t,\bx,\bv)-N(t,\bx,\bv')\right|
		\leq C \eps \left|\bv-\bv'\right| .
	\end{equation*}

	We now estimate $D_t M - D_t N$ to prove \eqref{estimate:N:2}. First, we rewrite $N$ by integrating by part twice the time integral in \eqref{formula:N} and find
	\[
	N(t,\bx,\bv)
	=
	M(t,\bx)
	-\eps\, D_t M(t,\bx)
	+
	\eps\,D_t\,M(0,\bx-t\bv)\, e^{-\frac{t}{\eps}}
	+
	\eps
	\int_0^t
	D_t^2M(s,\bx+(s-t)\bv)\,e^{\frac{s-t}{\eps}}\dD s\,.
	\]
	Then, we divide the previous relation by $\eps$ and replace $(M-N)/\eps$ by $D_t N$ according to \eqref{eq:N}, this yields
	\[
	D_t M(t,\bx)
	-D_t N(t,\bx,\bv)
	=
	D_tM(0,\bx-t\bv)\, e^{-\frac{t}{\eps}}
	+
	\int_0^t
	D_t^2M(s,\bx+(s-t)\bv)\,e^{\frac{s-t}{\eps}}\dD s\,.
	\]
	Taking the absolute value on both sides and using assumption \eqref{hyp:M} to bound the derivatives of $M$, we obtain the result
	\[
	\left|D_t M(t,\bx)
	-D_t N(t,\bx,\bv)\right|
	\,\leq\,4\,\|M\|_{W^{2,\infty}}\left(\eps+e^{-\frac{t}{\eps}}\right)\left(1+|\bv|^2\right) \,.
	\]
\end{proof}

In the next proposition, we provide uniform in $\eps$ moment estimates with respect to $\bx$, $\bv$ and $y$ for the solution $q^\eps$ to \eqref{eq:q:eps}. To obtain this result, we first propagate moments with respect to $\bv$ using the integrability assumption \eqref{hyp:Lambda+} on the tumbling kernel $\Lambda$, then deduce propagation of moments with respect to $\bx$ using standard techniques for transport equations. Moments with respect to $y$ are also estimated thanks assumption \eqref{hyp:Lambda+} coupled with the estimates for $N$ in Lemma \ref{estimate:N} and the bound on moments with respect to $\bv$. For later purposes, we use the moment estimates on $\bv$ and $y$ to prove that $\Tilde{Q}^\eps\left[q^\eps\right]$ also has second order moments with respect to $y$. To conclude, we obtain propagation of moments in $\bv$ for the solution $\bar{p}$ to \eqref{eq:p:lim}-\eqref{def:tumbling:op:lim} thanks to similar arguments.
\begin{proposition}\label{moment:estimate} Under assumptions \eqref{hyp:Lambda+} on $\Lambda$, \eqref{hyp:M} on $M$ and \eqref{hyp1:p:eps} on $\left(p^\eps_0\right)_{\eps>0}$ consider a sequence of strong solutions $\left(p^\eps\right)_{\eps>0}$ to \eqref{eq:p}-\eqref{def:tumbling:op} with initial data $\left(p^\eps_0\right)_{\eps>0}$. \\
There exists a positive constant $C$ such that for all $\eps>0$ and all $t\geq0$, the function $q^\eps$ defined by \eqref{def:q:eps} satisfies 
\begin{equation}\label{moments:v} 
\int_{\R^d\times V\times \R}
|\bv|^2\,q^\eps(t,\bx,\bv,y)\,
\dD \bx\dD \bv \dD y\,<\,C\left(t+1\right),
\end{equation}
and 
\begin{equation}\label{moments:x} 
	\int_{\R^d\times V\times \R}
	|\bx|\,q^\eps(t,\bx,\bv,y)\,
	\dD \bx\dD \bv \dD y\,<\,C\left(t^{\frac{3}{2}}+1\right).
\end{equation}
Furthermore, there exists $\eps_0>0$ such that for all $\eps$ between $0$ and $\eps_0$, it holds
\begin{equation}\label{moments:y}
	\int_{\R^d\times V\times \R}|y|^2
	q^\eps(t,\bx,\bv,y)\,
	\dD \bx\dD \bv \dD y
	\,\leq\, C\left(\eps^{-2}e^{-\frac{t}{\eps}}
	\,+\,
	\eps\,t\,+\,1\right),
\end{equation}
and 
\begin{equation}\label{moments:y:2}
	\int_{\R^d\times V\times \R}  |y|^2\,
	\Tilde{Q}^\eps\left[q^\eps\right](t,\bx,\bv,y)
	\dD \bx\dD \bv \dD y\,\leq\,
	C\left(t^{\frac{1}{2}}\eps^{-1}e^{-\frac{t}{2\eps}}\,+\,\eps^{-1}e^{-\frac{t}{2\eps}}+ \,t \,+\, 1\right).
\end{equation}
Furthermore, under assumption \eqref{hyp1:p} on $\bar{p}_0$, consider a strong solution $\bar{p}$ to \eqref{eq:p:lim}-\eqref{def:tumbling:op:lim} with initial data $\bar{p}_0$. There exists a positive constant $C$ such that
\begin{equation}\label{moments:v:p} 
	\int_{\R^d\times V}
	|\bv|^2\,\bar{p}(t,\bx,\bv)\,
	\dD \bx\dD \bv\,<\,C\left(t+1\right),\quad\forall\,t\in\R^+\,.
\end{equation}
\end{proposition}

\begin{proof}
	We first prove \eqref{moments:v}. We denote by $\left\||\bv|^2 p^\eps(t)\right\|_{L^1}$ the second order moments of $q^\eps$ with respect to $\bv$, that is
	\[
	\left\||\bv|^2 q^\eps(t)\right\|_{L^1}
	\,=\,
	\int_{\R^d\times V\times \R}
	|\bv|^2\,q^\eps(t,\bx,\bv,y)\,
	\dD \bx\dD \bv \dD y\,.
	\]
	To estimate $\left\||\bv|^2 q^\eps\right\|_{L^1}$, we will compute its time derivative. However, we first notice that $q^\eps$ and the solution $p^\eps$ to \eqref{eq:p} have the same second order moments with respect to $\bv$. Indeed, inverting \eqref{change:var}, we obtain 
	\[
	\left\||\bv|^2 q^\eps(t)\right\|_{L^1}
	\,=\,
	\int_{\R^d\times V\times \R}
	|\bv|^2\,p^\eps(t,\bx,\bv,m)\,
	\dD \bx\dD \bv \dD m\,.
	\]
	Therefore the time derivative of $\left\||\bv|^2 q^\eps\right\|_{L^1}$ is obtained multiplying equation \eqref{eq:p} by $|\bv|^2$ and integrating with respect to  $(\bx,\bv,m)\in \R^d\times V\times \R$, it yields
	\[
	\frac{\dD}{\dD t} \left\||\bv|^2 q^\eps(t)\right\|_{L^1}\,=\,
	\int  |\bv|^2
	\left(
	\,\partial_m \left[\,\frac{m-M}{\eps} \,p^\eps+\eps\,\partial_m p^\eps\right]
	\,-\, \bv\cdot\nabla_\bx p^\eps+Q^\eps\left[p^\eps\right]\right)
	\dD \bx\dD \bv \dD m\,.
	\]
	Let us simplify the preceding right-hand side: both contributions of the transport and the Fokker-Planck operators cancel since it holds
	\[
	\int   |\bv|^2\,\bv\cdot\nabla_\bx p^\eps
	\dD \bx\dD \bv \dD m\,= \,\int   \nabla_\bx\cdot\left(|\bv|^2\bv\, p^\eps\right)
	\dD \bx\dD \bv \dD m\,= \,0\,,
	\]
	and
	\[
	\int |\bv|^2\,\partial_m \left[\,\frac{m-M}{\eps} \,p^\eps+\eps\,\partial_m p^\eps\right]
	\dD \bx\dD \bv \dD m = \int \partial_m \left[\,|\bv|^2\left(\frac{m-M}{\eps} \,p^\eps+\eps\,\partial_m p^\eps\right)\right]
	\dD \bx\dD \bv \dD m\,=\,0\,.
	\]
	 Hence, the time derivative of $\left\||\bv|^2 q^\eps\right\|_{L^1}$ rewrites
	\begin{align*}
	\frac{\dD}{\dD t} \left\||\bv|^2 q^\eps(t)\right\|_{L^1}\,=\,
	&\int|\bv|^2\,
	\Lambda\left(\frac{1}{\eps}\left(m-M(t,\bx)\right),\bv,\bv'\right)p^\eps\left(t,\bx,\bv',m\right)  \dD \bx\dD \bv \dD \bv'\dD m\,\\[0.8em]
	&-
	\int|\bv|^2\,
	\Lambda\left(\frac{1}{\eps}\left(m-M(t,\bx)\right),\bv',\bv\right)p^\eps\left(t,\bx,\bv,m\right)  \dD \bx\dD \bv \dD \bv'\dD m
	\,.
	\end{align*}
	Since the second term in the right-hand side is negative, we bound it by $0$ and hence deduce
	\begin{equation*}
		\frac{\dD}{\dD t}\left\||\bv|^2 q^\eps(t)\right\|_{L^1}\,\leq\,
		\int|\bv|^2\,
		\Lambda\left(\frac{1}{\eps}\left(m-M(t,\bx)\right),\bv,\bv'\right)p^\eps\left(t,\bx,\bv',m\right)  \dD \bx\dD \bv \dD \bv'\dD m
		\,.
	\end{equation*}
Then, we estimate the supremum of $\int|\bv|^2\,
	\Lambda\left(\frac{1}{\eps}\left(m-M(t,\bx)\right),\bv,\bv'\right)\dD \bv$ over all $(\bx,\bv',m)\in \R^d\times V\times \R$ thanks to assumption \eqref{hyp:Lambda+}. Thanks to conservation of mass for \eqref{eq:p}, we obtain
	\begin{equation*}
		\frac{\dD}{\dD t} \left\||\bv|^2 q^\eps(t)\right\|_{L^1}\,\leq\,C
		\,.
	\end{equation*}
	We integrate the previous inequality with respect to $t$ and use assumption \eqref{hyp1:p:eps} to bound $\left\||\bv|^2 q^\eps_0\right\|_{L^1}$, which yields estimate \eqref{moments:v}:
	\[
	\left\||\bv|^2 q^\eps(t)\right\|_{L^1}\,\leq\,C(t+1)\,.
	\]
	
	The proof of estimate \eqref{moments:v:p} for moments in $\bv$ of $\bar{p}$ follows the same lines as the one detailed above for $q^\eps$. Hence we do not detail it here.\\
	
We now proceed to the proof of \eqref{moments:x}. We denote by $\left\||\bx|\, q^\eps(t)\right\|_{L^1}$ the first order moments of $q^\eps$ with respect to $\bx$, that is
\[
\left\||\bx|\, q^\eps(t)\right\|_{L^1}
\,=\,
\int_{\R^d\times V\times \R}
|\bx|\,q^\eps(t,\bx,\bv,y)\,
\dD \bx\dD \bv \dD y\,.
\]
To estimate $\left\||\bx|\, q^\eps\right\|_{L^1}$, we compute its time derivative. Again, $q^\eps$ and the solution $p^\eps$ to \eqref{eq:p} have same moments with respect to $\bx$. Indeed, inverting \eqref{change:var}, we obtain 
\[
\left\||\bx|\, q^\eps(t)\right\|_{L^1}
\,=\,
\int_{\R^d\times V\times \R}
|\bx|\,p^\eps(t,\bx,\bv,m)\,
\dD \bx\dD \bv \dD m\,.
\]
Therefore the time derivative of $\left\||\bx|\, q^\eps\right\|_{L^1}$ is obtained multiplying equation \eqref{eq:p} by $|\bx|$ and integrating with respect  $(\bx,\bv,m)\in \R^d\times V\times \R$, it yields
\[
\frac{\dD}{\dD t} \left\||\bx|\, q^\eps(t)\right\|_{L^1}\,=\,
\int  |\bx|
\left(
\,\partial_m \left[\,\frac{m-M}{\eps} \,p^\eps+\eps\,\partial_m p^\eps\right]
\,-\, \bv\cdot\nabla_\bx p^\eps+Q^\eps\left[p^\eps\right]\right)
\dD \bx\dD \bv \dD m\,.
\]
To simplify the right-hand side, we first notice that, as for moments with respect to the velocity variable, the contribution the Fokker-Planck operator cancels. The same holds for the contribution of the tumbling operator since $\int  Q^\eps\left[p^\eps\right]\dD \bv \,=\,0$. Hence, the previous relation rewrites
\[
\frac{\dD}{\dD t} \left\||\bx|\, q^\eps(t)\right\|_{L^1}\,=\,-
\int  |\bx|\,
 \bv\cdot\nabla_\bx p^\eps
\dD \bx\dD \bv \dD m\,.
\]
After integrating by part with respect to $\bx$ in the integral, applying the Cauchy-Schwarz inequality and using mass conservation for \eqref{eq:p}, we obtain
\[
\frac{\dD}{\dD t} \left\||\bx|\, q^\eps(t)\right\|_{L^1}\,\leq\,
\left(
\int 
|\bv|^2 p^\eps
\dD \bx\dD \bv \dD m\right)^\frac{1}{2}\,.
\]
We bound the right-hand side thanks to \eqref{moments:v} and deduce
\[
\frac{\dD}{\dD t} \left\||\bx|\, q^\eps(t)\right\|_{L^1}\,\leq\,
C\left(
1+t\right)^\frac{1}{2}\,.
\]
Estimate \eqref{moments:x} is obtained integrating the latter with respect to $t$ and applying assumption \eqref{hyp1:p:eps} to bound $\left\||\bx|\, q^\eps_0\right\|_{L^1}$.\\

We now derive \eqref{moments:y}. We denote by $\left\||y|^2 q^\eps(t)\right\|_{L^1}$ the second order moments of $q^\eps$ with respect to $y$, that is
\[
\left\||y|^2 q^\eps(t)\right\|_{L^1}
\,=\,
\int_{\R^d\times V\times \R}
|y|^2
q^\eps(t,\bx,\bv,y)\,
\dD \bx\dD \bv \dD y\,.
\]
To estimate $\left\||y|^2 q^\eps(t)\right\|_{L^1}$, we compute its time derivative, obtained multiplying equation \eqref{eq:q:eps} by $|y|^2$ and integrating with respect  $(\bx,\bv,y)\in \R^d\times V\times \R$, it yields
\[
\frac{\dD}{\dD t} \left\||y|^2 q^\eps(t)\right\|_{L^1}\,=\,
\int  |y|^2
\left(
\frac{1}{\eps}\,\partial_y \left[\,y \,q^\eps+\partial_y q^\eps\right]
\,-\, \bv\cdot\nabla_\bx q^\eps+\Tilde{Q}^\eps\left[q^\eps\right]\right)
\dD \bx\dD \bv \dD y\,.
\]
Let us simplify the right-hand side: first, the transport operator has no contribution, that is
\[
\int  |y|^2\,\bv\cdot\nabla_\bx q^\eps
\dD \bx\dD \bv \dD y\,= \,\int   \nabla_\bx\cdot\left(|y|^2\,\bv \,q^\eps\right)
\dD \bx\dD \bv \dD y\,= \,0\,.
\]
Second, we integrate by part with respect to $y$ the contribution associated to the Fokker-Planck operator. Since $q^\eps$ as mass $1$, it yields
\[
\int|y|^2\,\partial_y \left[\,y \,q^\eps+\partial_y q^\eps\right]
\dD \bx\dD \bv \dD y = - 2\left\||y|^2 q^\eps(t)\right\|_{L^1}+2\,.
\]
Therefore the time derivative of $\left\||y|^2 q^\eps\right\|_{L^1}$ rewrites
\[
\frac{\dD}{\dD t} \left\||y|^2 q^\eps(t)\right\|_{L^1}\,=\,
\frac{2}{\eps} \left(1- \left\||y|^2 q^\eps(t)\right\|_{L^1}\right) +\cJ\,,
\]
where $\cJ$ denotes the contribution of the tumbling operator, that is
\begin{equation}\label{def:Q}
	\cJ\,=\,
\int  |y|^2\,
\Tilde{Q}^\eps\left[q^\eps\right]
\dD \bx\dD \bv \dD y\,.
\end{equation}
To estimate $\cJ$, we rewrite it as 
$\cJ\,=\,
\cJ_1-\cJ_2$
where $\cJ_1$ and $\cJ_2$ are given by
\begin{equation*}
	\left\{
	\begin{array}{l}
		\displaystyle \cJ_1 = 	\int
		|y|^2\,
		\Lambda\left(y'-D_t N(t,\bx,\bv'),\bv,\bv'\right)q^\eps\left(t,\bx,\bv',y'\right)  \dD \bx\dD \bv \dD \bv'\dD y\,
		,\\[1.5em]
		\displaystyle
		\cJ_2\,=\,\int
		|y|^2\,
		\Lambda\left(y-D_t N(t,\bx,\bv),\bv',\bv\right)q^\eps\left(t,\bx,\bv,y\right)	\dD \bx\dD \bv \dD \bv'\dD y
		\,,
	\end{array}
	\right.
\end{equation*}
and where we used the shorthand notation 
\[y' = y + \left(N(t,\bx,\bv)-N(t,\bx,\bv')\right)/\eps\,.\]
In $\cJ_1$, we perform the change of variable $y \leftarrow y'$ and then invert variables $\bv$ and $\bv'$, this yields
\[
\cJ_1 = 	\int
\left|y + \frac{N(t,\bx,\bv)-N(t,\bx,\bv')}{\eps}\right|^2
\Lambda\left(y-D_t N(t,\bx,\bv),\bv',\bv\right)q^\eps\left(t,\bx,\bv,y\right)	\dD \bx\dD \bv \dD \bv'\dD y\,.
\]
We rewrite $\cJ$ substituting $\cJ_1$ thanks to the previous relation, this yields
\[
\cJ\,=\,
\int
\left(\left|y + \frac{N(t,\bx,\bv)-N(t,\bx,\bv')}{\eps}\right|^2-\left|y\right|^2\right)
\Lambda\left(y-D_t N(t,\bx,\bv),\bv',\bv\right)q^\eps\left(t,\bx,\bv,y\right)\dD \bx\dD \bv \dD \bv'\dD y.
\]
To estimate $\cJ$, we take the absolute value inside the preceding integral and bound the difference between squares as follows
\begin{align*}
\left|\left|y + \frac{N(t,\bx,\bv)-N(t,\bx,\bv')}{\eps}\right|^2-\left|y\right|^2\right|
&\leq
2\left|y\right|
\left|\frac{N(t,\bx,\bv)-N(t,\bx,\bv')}{\eps}\right|
+\left|\frac{N(t,\bx,\bv)-N(t,\bx,\bv')}{\eps}\right|^2\\[0.8em]
&\leq C\left(|y|\,|\bv-\bv'|+|\bv-\bv'|^2\right)\,,
\end{align*}
where the last line is obtained thanks to \eqref{estimate:N:1} in Lemma \ref{estimate:N} to estimate $N$. Hence, we obtain
\[
\cJ\,\leq\,
C
\int
\left(|y|\,|\bv-\bv'|+|\bv-\bv'|^2\right)
\Lambda\left(y-D_t N(t,\bx,\bv),\bv',\bv\right)q^\eps\left(t,\bx,\bv,y\right)	\dD \bx\dD \bv \dD \bv'\dD y\,.
\]
We apply the Cauchy-Schwarz inequality to estimate the product $|y|\,|\bv-\bv'|$ and we bound $|\bv-\bv'|^2$ by $2(|\bv|^2+|\bv'|^2)$, which yields
\[
\cJ\,\leq\,
C\left(\cJ_3^{\frac{1}{2}}\,\cJ_4^{\frac{1}{2}} + \cJ_4\right)\,,
\]
where $\cJ_3$ and $\cJ_4$ are defined as
\begin{equation*}
	\left\{
	\begin{array}{l}
		\displaystyle \cJ_3 = 	\int |y|^2
		\Lambda\left(y-D_t N(t,\bx,\bv),\bv',\bv\right)q^\eps\left(t,\bx,\bv,y\right)	\dD \bx\dD \bv \dD \bv'\dD y
		,\\[1.5em]
		\displaystyle
		\cJ_4\,=\,\int \left(|\bv|^2+|\bv'|^2\right)
		\Lambda\left(y-D_t N(t,\bx,\bv),\bv',\bv\right)q^\eps\left(t,\bx,\bv,y\right)	\dD \bx\dD \bv \dD \bv'\dD y
		\,.
	\end{array}
	\right.
\end{equation*}
In both $\cJ_3$ and $\cJ_4$, we integrate $\Lambda$ with respect to $\bv'$ and estimate $\int(1+|\bv'|^2)
\Lambda\left(y-D_t N(t,\bx,\bv),\bv',\bv\right)\dD \bv'$ by its supremum over all $(\bx,\bv,y)\in \R^d\times V\times \R$, which is finite according to assumption \eqref{hyp:Lambda+}. Since the total mass of $q^\eps$ is conserved, we deduce 
\[
\cJ\,\leq\,
C\left(\left(\int |y|^2
q^\eps	\dD \bx\dD \bv\dD y\right)^{\frac{1}{2}}\left(\int |\bv|^2
q^\eps	\dD \bx\dD \bv\dD y+1\right)^{\frac{1}{2}} + \int |\bv|^2
q^\eps	\dD \bx\dD \bv\dD y \,+\, 1\right)\,.
\]
We bound moments with respect to $\bv$ thanks to \eqref{moments:v} in Proposition \ref{moment:estimate} and deduce the following estimate for $\cJ$
\begin{equation}\label{estimate:Q}
\cJ\,\leq\,
C\left((t+1)^{\frac{1}{2}}\left(\int |y|^2
q^\eps	\dD \bx\dD \bv\dD y\right)^{\frac{1}{2}} + \,t \,+\, 1\right)\,.
\end{equation}
After applying Young inequality, we obtain 
\begin{equation*}
	\cJ\,\leq\,
	C\left(\left\||y|^2 q^\eps(t)\right\|_{L^1}\,+ \,t \,+\, 1\right)\,.
\end{equation*}
Therefore, we get the following differential inequality for $\left\||y|^2 q^\eps\right\|_{L^1}$
\[
\frac{\dD}{\dD t} \left\||y|^2 q^\eps(t)\right\|_{L^1}\,\leq\,\left(C-\frac{2}{\eps}\right)\left\||y|^2 q^\eps(t)\right\|_{L^1}
+\frac{1}{\eps} +C(t+1)\,.
\]
We multiply this inequality by $\exp{\left((2/\eps-C)\,t\right)}$ and integrate with respect to the time variable, which yields
\[
\left\||y|^2 q^\eps(t)\right\|_{L^1}\,\leq\, \left\||y|^2 q^\eps_0\right\|_{L^1}\,e^{\left(C-\frac{2}{\eps}\right)t}
\,+\,\frac{C\eps+1}{2-C\eps}\left(1-e^{\left(C-\frac{2}{\eps}\right)t}\right)
\,+\,C\int_0^t s\,e^{\left(C-\frac{2}{\eps}\right)(t-s)}\,\dD s\,.
\]
We compute the time integral on the right-hand side thanks to an integration by part and obtain 
\begin{align*}
\left\||y|^2 q^\eps(t)\right\|_{L^1}\,\leq\, \left\||y|^2 q^\eps_0\right\|_{L^1}\,e^{\left(C-\frac{2}{\eps}\right)t}
\,+\,\frac{C\eps+1}{2-C\eps}&\left(1-e^{\left(C-\frac{2}{\eps}\right)t}\right)\\[0.8em]
&\,+\,
\frac{C\eps\,t}{2-C\eps}\,-\,
\frac{C\eps^2}{\left(2-C\eps\right)^2}\left(1-e^{\left(C-\frac{2}{\eps}\right)t}\right)
\,.
\end{align*}
Taking $\eps\leq1/C$, we obtain
\[
\left\||y|^2 q^\eps(t)\right\|_{L^1}\,\leq\, \left\||y|^2 q^\eps_0\right\|_{L^1}\,e^{-\frac{t}{\eps}}
\,+\,2
\,+\,
C\eps\,t\,.
\]
To estimate $\left\||y|^2 q^\eps_0\right\|_{L^1}$, we invert the change of variable \eqref{change:var}  with $t=0$ and obtain
\[
\left\||y|^2 q^\eps(t)\right\|_{L^1}\,\leq\, e^{-\frac{t}{\eps}}\int_{\R^d\times V\times \R}
\left|\frac{m-M(0,\bx)}{\eps}\right|^2\,p^\eps_0(\bx,\bv,m)\,
\dD \bx\dD \bv \dD m
\,+\,2
\,+\,
C\eps\,t\,.
\]
We estimate the integral on the right-hand side thanks to assumption \eqref{hyp1:p:eps} to bound the moments of $p^\eps$ with respect to $m$, and assumption \eqref{hyp:M} to bound $M(0,\bx)$. This yields the expected result
\[
\int_{\R^d\times V\times \R}
|y|^2
q^\eps(t,\bx,\bv,y)\,
\dD \bx\dD \bv \dD y
\,\leq\, C\left(\eps^{-2}e^{-\frac{t}{\eps}}
\,+\,
\eps\,t\,+\,1\right).
\]
To conclude this proof, we derive estimate \eqref{moments:y:2} as a straightforward consequence of our previous computations. Indeed, according to  \eqref{estimate:Q}, it holds
\begin{equation*}
	\cJ\,\leq\,
	C\left((t+1)^{\frac{1}{2}}\left(\int |y|^2
	q^\eps	\dD \bx\dD \bv\dD y\right)^{\frac{1}{2}} + \,t \,+\, 1\right),
\end{equation*}
where $\cJ$ is defined by \eqref{def:Q}. We bound  $\ds\int |y|^2
q^\eps	\dD \bx\dD \bv\dD y$ in the right-hand side thanks to \eqref{moments:y} in Proposition \ref{moment:estimate}, which yields
\begin{equation*}
	\cJ\,\leq\,
	C\left((t+1)^{\frac{1}{2}}\left(\eps^{-2}e^{-\frac{t}{\eps}}
	\,+\,
	\eps\,t\,+\,1\right)^{\frac{1}{2}} + \,t \,+\, 1\right).
\end{equation*}
Then, simple computations yield the result
\begin{equation*}
	\cJ\,\leq\,
	C\left(t^{\frac{1}{2}}\eps^{-1}e^{-\frac{t}{2\eps}}\,+\,\eps^{-1}e^{-\frac{t}{2\eps}}+ \,t \,+\, 1\right).
\end{equation*}
\end{proof}
\section{Convergence of the re-scaled distribution}
\label{sec:cv:estimates}
This section is dedicated to the convergence of $q^\eps$ defined by \eqref{def:q:eps} toward $\bar{p}\,M$, where $\bar{p}$ is the solution to \eqref{eq:p:lim}-\eqref{def:tumbling:op:lim} and $\cM$ stands for the standard Maxwellian and is given in \eqref{Maxwellian}. We proceed in two steps: first, we prove the convergence of $q^\eps$ towards its Maxwellian local equilibrium, that is
\[
q^\eps(t,\bx,\bv,y)\underset{\eps\rightarrow 0}{\sim}\bar{q}^\eps(t,\bx,\bv)\cM(y)\,,
\]
where the marginal $\bar{q}^\eps$ is given by \eqref{def:q:eps:bar}. Second, we use the preceding convergence to deduce the convergence of $\bar{q}^\eps$ toward $\bar{p}$.\\

To derive the convergence of $q^\eps$ towards its Maxwellian local equilibrium, we estimate the following relative energy along the trajectories of equation \eqref{eq:q:eps}
\[
E\left[q^\eps\right]
\,=\,
\int_{\R^d\times V\times \R}
q^\eps(t,\bx,\bv,y)
\log{\left(\frac{q^\eps(t,\bx,\bv,y)}{\cM(y)}\right)}\dD \bx\dD \bv \dD y\,,
\]
where $\cM$ is given by \eqref{Maxwellian}. We prove that the dissipation rate of this quantity controls the distance between $q^\eps$ and $\bar{q}^\eps\cM$. We obtain the following convergence estimate: 
\begin{proposition}\label{energy:estimate} Under assumptions \eqref{hyp:Lambda+}-\eqref{hyp:Lambda-} on $\Lambda$, \eqref{hyp:M} on $M$ and \eqref{hyp1:p:eps} on $\left(p^\eps_0\right)_{\eps>0}$ consider a sequence of strong solutions $\left(p^\eps\right)_{\eps>0}$ to \eqref{eq:p}-\eqref{def:tumbling:op} with initial data $\left(p^\eps_0\right)_{\eps>0}$. \\
There exist two positive constants $C$ and $\eps_0$ such that for all $\eps$ between $0$ and $\eps_0$, it holds
\[
\int_0^t \left\|\,
q^\eps(s)\,-\,\ols{q}^\eps(s)\,
\mathcal{M}\,
\right\|^2_{L^1
	\left(
	\R^d\times V\times \R
	\right)
}\dD s\,\leq\,\eps\,E[q^\eps_0] \,+\,
C\,\eps\left(t^2 \,+\, 1\right)\,,\quad\forall\,t \in\R^+\,.
\]
\end{proposition}
\begin{proof} To compute the time derivative of the relative energy $E[q^\eps]$, we multiply equation \eqref{eq:q:eps} by $\log{\left(q^\eps/\cM\right)}$ and integrate with respect to $\bx$, $\bv$ and $y$. Since mass is conserved along the trajectories of \eqref{eq:q:eps}, we obtain
\[
\frac{\dD}{\dD t} E[q^\eps]\,=\,
\int  \log{\left(\frac{q^\eps}{\cM}\right)}
\left(
 \frac{1}{\eps}\,\partial_y \left[\,y \,q^\eps+\partial_y q^\eps\right]
\,-\, \bv\cdot\nabla_\bx q^\eps+\Tilde{Q}^\eps\left[q^\eps\right]\right)
\dD \bx\dD \bv \dD y\,.
\]
We first simplify the previous right-hand side: since $\cM$ does not depend on $\bx$, the transport operator has no contribution, that is
\[
\int  \log{\left(\frac{q^\eps}{\cM}\right)}\bv\cdot\nabla_\bx q^\eps
\dD \bx\dD \bv \dD y\,= \,0\,.
\]
Furthermore, the Fokker-Planck operator has a signed contribution. Indeed, an integration by part yields
\[
\int  \log{\left(\frac{q^\eps}{\cM}\right)}\partial_y \left[\,y \,q^\eps+\partial_y q^\eps\right]
\dD \bx\dD \bv \dD y = - I[q^\eps]\,,
\]
where the Fisher information is defined as follows
\[
I[q^\eps]\,=\,\int  \left|\partial_y\log{\left(\frac{q^\eps}{\cM}\right)}\right|^2 q^\eps
\dD \bx\dD \bv \dD y \,.
\]
Consequent to the last two remarks, we obtain
\[
\frac{\dD}{\dD t} E[q^\eps] \,+\, \frac{1}{\eps}\,I[q^\eps]\,=\,
\int  \log{\left(\frac{q^\eps}{\cM}\right)}\Tilde{Q}^\eps\left[q^\eps\right]
\dD \bx\dD \bv \dD y\,.
\]
To estimate the contribution of the tumbling operator, we decompose it as follows
\[
\int  \log{\left(\frac{q^\eps}{\cM}\right)}\Tilde{Q}^\eps\left[q^\eps\right]
\dD \bx\dD \bv \dD y\,=\,
\frac{1}{2}\,\cJ+ \cK\,,
\]
where $\cJ$ and $\cK$ are given by
\begin{equation*}
	\left\{
	\begin{array}{l}
		\displaystyle \cJ = \int  |y|^2\,
		\Tilde{Q}^\eps\left[q^\eps\right]
		\dD \bx\dD \bv \dD y\,
		,\\[1.5em]
		\displaystyle
		\cK\,=\,\int  \log{\left(q^\eps\right)}\,\Tilde{Q}^\eps\left[q^\eps\right]
		\dD \bx\dD \bv \dD y
		\,.
	\end{array}
	\right.
\end{equation*}
We use dissipation properties of the tumbling kernel to estimate $\cK$, whereas $\cJ$ was already estimated in Proposition \ref{moment:estimate}, estimate \eqref{moments:y:2}.\\

We first estimate $\cK$, which is given by
\begin{align*}
	\cK = 	\int&
		\Lambda\left(y'-D_t N(t,\bx,\bv'),\bv,\bv'\right)q^\eps\left(t,\bx,\bv',y'\right)\log{\left(q^\eps\left(t,\bx,\bv,y\right)\right)}	\dD \bx\dD \bv \dD \bv'\dD y\\[0.5em]
		&- \int
		\Lambda\left(y-D_t N(t,\bx,\bv),\bv',\bv\right)q^\eps\left(t,\bx,\bv,y\right)\log{\left(q^\eps\left(t,\bx,\bv,y\right)\right)}	\dD \bx\dD \bv \dD \bv'\dD y
		\,,
\end{align*}
where we used the shorthand notation $y' = y + \left(N(t,\bx,\bv)-N(t,\bx,\bv')\right)/\eps$. In the second term of the former right-hand side, we perform the change of variable $y \leftarrow y + \left(N(t,\bx,\bv)-N(t,\bx,\bv')\right)/\eps$ and invert $\bv$ and $\bv'$, which yields
\[
\cK = 	\int
\Lambda\left(y'-D_t N(t,\bx,\bv'),\bv,\bv'\right)q^\eps\left(t,\bx,\bv',y'\right)\log{\left(\frac{q^\eps\left(t,\bx,\bv,y\right)}{q^\eps\left(t,\bx,\bv',y'\right)}\right)}	\dD \bx\dD \bv \dD \bv'\dD y\,.
\]
We estimate the logarithm in the right-hand side thanks to the inequality $\log{\left(x\right)}\leq x$, this yields
\[
\cK\,\leq\,
\int
\Lambda\left(y-D_t N(t,\bx,\bv),\bv,\bv'\right)q^\eps\left(t,\bx,\bv,y\right)	\dD \bx\dD \bv \dD \bv'\dD y\,,
\]
where we also used that $y'-D_t N(t,\bx,\bv') = y-D_t N(t,\bx,\bv)$. Thanks to assumption \eqref{hyp:Lambda-}, we estimate $\int
\Lambda\left(y-D_t N(t,\bx,\bv),\bv,\bv'\right)\dD \bv'$ by its supremum over all $(\bx,\bv,y)\in \R^d\times V\times \R$. Since the mass of $q^\eps$ is conserved, this yields
\[
\cK\,\leq\,C\,.
\]

To estimate $\cJ$, we apply \eqref{moments:y:2}, which ensures
\begin{equation*}
	\cJ\,\leq\,
	C\left(t^{\frac{1}{2}}\eps^{-1}e^{-\frac{t}{2\eps}}\,+\,\eps^{-1}e^{-\frac{t}{2\eps}}+ \,t \,+\, 1\right).
\end{equation*}
Gathering estimates for $\cJ$ and $\cK$, we deduce
\[
\frac{\dD}{\dD t} E[q^\eps] \,+\, \frac{1}{\eps}I[q^\eps]\,\leq\,
	C\left(t^{\frac{1}{2}}\eps^{-1}e^{-\frac{t}{2\eps}}\,+\,\eps^{-1}e^{-\frac{t}{2\eps}}+ \,t \,+\, 1\right).
\]
Integrating the inequality between $0$ and $t$, we obtain
\[
 \frac{1}{\eps}\int_0^t I[q^\eps(s)]\,\dD s\,\leq\,E[q^\eps_0] \,-\,E[q^\eps(t)] \,+\,
C\left(t^2 \,+\, 1\right),
\]
where we used
\[\eps^{-1}\int_0^t s^{\frac{1}{2}}e^{-\frac{s}{2\eps}}\,\dD s\leq\sqrt{\eps}\int_0^{+\infty} \sigma^{\frac{1}{2}}e^{-\frac{\sigma}{2}}\,\dD s\leq C.\] On the one hand, the Fisher information $I[q^\eps(s)]$ controls the $L^1$-norm between $q^\eps(s)$ and $\bar{q}^\eps(s)\cM$ since we have
\[
\left\|\,
q^\eps(s)\,-\,\ols{q}^\eps(s)\,
\mathcal{M}\,
\right\|^2_{L^1
	\left(
	\R^d\times V\times \R
	\right)
}
\,\leq\,
2\,
\int_{\R^d\times V\times \R}
q^\eps(s)\log\left(
\frac{q^\eps(s)}{
 \ols{q}^\eps(s)\,\mathcal{M}
}
\right)\dD \bx\dD \bv\dD y
\,\leq
\,
I
\left[\,
q^\eps(s)
\right]\,,
\]
where first and second inequalities correspond respectively to Csizár-Kullback and Gaussian logarithmic Sobolev (see \cite{Fathi/Indrei/Ledoux}) inequalities. On the other hand, we have 
\begin{align*}
-\,E[q^\eps(t)]
\,&=\,-
\int
q^\eps(t,\bx,\bv,y)\left(
\log{\left(\frac{q^\eps(t,\bx,\bv,y)}{\cM(y)\cM(|\bv|)\cM(|\bx|^{\frac{1}{2}})}\right)}
-\frac{1}{2}|\bv|^2-\frac{1}{2}|\bx|-\log{\left(2\pi\right)}
\right)
\dD \bx\dD \bv \dD y\\
&\leq\,
\int
q^\eps(t,\bx,\bv,y)\left(
\log{\left(\int\cM(|\bv|)\cM(|\bx|^{\frac{1}{2}})\dD \bx\dD \bv)\right)}
+\frac{1}{2}|\bv|^2+\frac{1}{2}|\bx|+\log{\left(2\pi\right)}
\right)
\dD \bx\dD \bv \dD y\\
&\leq\,C\left(t^{\frac{3}{2}}+1\right)
\,,
\end{align*}
where we used Jensen inequality with the measure $\cM(y)\cM(|\bv|)\cM(|\bx|^{\frac{1}{2}})$ to obtain the second line and estimates \eqref{moments:v}-\eqref{moments:x} to bound moments at the last line. Gathering these two estimates, we obtain the result
\[
\int_0^t \left\|\,
q^\eps(s)\,-\,\ols{q}^\eps(s)\,
\mathcal{M}\,
\right\|^2_{L^1
	\left(
	\R^d\times V\times \R
	\right)
}\dD s\,\leq\,\eps\,E[q^\eps_0] \,+\,
C\,\eps\left(t^2 +\, 1\right).
\]
\end{proof}
To complete this section, we estimate the distance between $\bar{q}^\eps$ given in \eqref{def:q:eps:bar} and the solution $\bar{p}$ to limiting model \eqref{eq:p:lim}-\eqref{def:tumbling:op:lim}. The key point consists in taking advantage of the dissipation properties of the limiting tumbling kernel $Q$, which has a signed contribution to the error between $\bar{q}^\eps$ and $\bar{p}$. Then, remaining contributions are error terms that we estimate thanks to our previous convergence estimates for $N$ towards $M$ and $q^\eps$ towards $\ols{q}^\eps \mathcal{M}$. 
\begin{proposition}\label{cv:q:p:bar}
Under assumptions \eqref{hyp:Lambda+}-\eqref{hyp:Lambda-} on $\Lambda$, \eqref{hyp:M} on $M$, \eqref{hyp1:p:eps} on $\left(p^\eps_0\right)_{\eps>0}$ and \eqref{hyp1:p} on $\bar{p}_0$, consider a sequence of strong solutions $\left(p^\eps\right)_{\eps>0}$ to \eqref{eq:p}-\eqref{def:tumbling:op} with initial data $\left(p^\eps_0\right)_{\eps>0}$ and a strong solution $\bar{p}$ to \eqref{eq:p:lim}-\eqref{def:tumbling:op:lim} with initial data $\bar{p}_0$. \\
There exist two positive constants $C$ and $\eps_0$ such that for all $\eps$ between $0$ and $\eps_0$, it holds
\[
\left\|\bar{p}(t)-\bar{q}^\eps(t)\right\|_{L^1}
\,\leq\,
\left\|\bar{p}_0-\bar{q}^\eps_0\right\|_{L^1}
+\,
C\left(\sqrt{\eps\, t}\,E[q^\eps_0]^{\frac{1}{2}} \,+\,
\sqrt{\eps\, t}\,\left(t \,+\, 1\right)
+
\eps^2\right)
\,,\quad\forall \,t \in \R^+\,.
\]
\end{proposition}
\begin{proof}
To simplify notations, we omit the dependence with respect to $(t,\bx)\in\R^+\times\R^d$ when the context is clear. To estimate $\left\|\bar{p}(t)-\bar{q}^\eps(t)\right\|_{L^1\left(\R^d\times V\right)}$, we first compute the equation solved by $\bar{q}^\eps$, which is obtained integrating \eqref{eq:q:eps} with respect to $y\in\R$, that is
\begin{equation}
	\label{eq:q:eps:bar}
	\ds\frac{\partial \bar{q}^\eps}{\partial t}\,+\,\bv\cdot\nabla_\bx \bar{q}^\eps
	\,=\,\int_{\R}\Tilde{Q}^\eps\left[q^\eps\right]\dD y\,,
\end{equation}
where 
\[
\int\Tilde{Q}^\eps\left[q^\eps\right]\dD y
\,=\,
\int
\Lambda\left(y'-D_t N(\bv'),\bv,\bv'\right)q^\eps\left(\bv',y'\right)-
\Lambda\left(y-D_t N(\bv),\bv',\bv\right)q^\eps\left(\bv,y\right) \dD \bv'\dD y
\,,
\]
and where we used the shorthand notation 
\[y' = y + \left(N(t,\bx,\bv)-N(t,\bx,\bv')\right)/\eps\,.\]
We perform the change of variable $y \leftarrow y + \left(N(t,\bx,\bv)-N(t,\bx,\bv')\right)/\eps$ in the first term of the previous right-hand side and deduce
\[
\int\Tilde{Q}^\eps\left[q^\eps\right]\dD y
\,=\,
\int
\Lambda\left(y-D_t N(\bv'),\bv,\bv'\right)q^\eps\left(\bv',y\right)-
\Lambda\left(y-D_t N(\bv),\bv',\bv\right)q^\eps\left(\bv,y\right) \dD \bv'\dD y
\,.
\]
Then, we compute the time derivative of  $\left\|\bar{p}(t)-\bar{q}^\eps(t)\right\|_{L^1\left(\R^d\times V\right)}$ by multiplying the difference between \eqref{eq:p:lim} and \eqref{eq:q:eps:bar} by $\sgn{\left(\bar{p}(t)-\bar{q}^\eps(t)\right)}$ and integrating with respect to $\left(\bx,\bv\right)\in\R^d\times V$; we obtain
\[
\frac{\dD}{\dD t}
\left\|\bar{p}-\bar{q}^\eps\right\|_{L^1\left(\R^d\times V\right)}
\,=\,
\int_{\R^d\times V}
\sgn{\left(\bar{p}-\bar{q}^\eps\right)}\left(
Q\left[\bar{p}\right] 
-
\int_{\R}\Tilde{Q}^\eps\left[q^\eps\right]\dD y
\right)\dD \bx\dD \bv\,.
\]
We decompose the right-hand side as follows
\[
\frac{\dD}{\dD t}
\left\|\bar{p}-\bar{q}^\eps\right\|_{L^1\left(\R^d\times V\right)}
\,=\,
\cL_{11}-\cL_{12}
\,+\,
\cL_{21}-\cL_{22}
\,+\,
\cL_{31}-\cL_{32}
\,,
\]
and carefully detail each contribution. The first term $\cL_{11}-\cL_{12}$ corresponds to the error due to the difference between the external signal $M$ and $N$ defined in \eqref{eq:N}, that is
\begin{equation*}
	\left\{
	\begin{array}{l}
		\displaystyle \cL_{11} =	\int
		\sgn{\left(\bar{p}-\bar{q}^\eps\right)}(\bv)\left(
		\Lambda\left(y-D_t M(\bv'),\bv,\bv'\right)-\Lambda\left(y-D_t N(\bv'),\bv,\bv'\right)\right)\bar{p}\left(\bv'\right)\cM(y)\,
		\dD \bx\dD \bv \dD \bv'\dD y\,
		,\\[1.5em]
		\displaystyle \cL_{12} = 	\int\sgn{\left(\bar{p}-\bar{q}^\eps\right)}(\bv)\left(
		\Lambda\left(y-D_t M(\bv),\bv',\bv\right)-\Lambda\left(y-D_t N(\bv),\bv',\bv\right)\right)\bar{p}\left(\bv\right)\cM(y)\,
		\dD \bx\dD \bv \dD \bv'\dD y\,,
	\end{array}
	\right.
\end{equation*}
we estimate $\cL_{11}$ and $\cL_{12}$ thanks to the convergence of $N$ towards $M$ ensured in Lemma \ref{estimate:N}. The second term $\cL_{21}-\cL_{22}$ corresponds to the propagation of the $L^1$-norm for the limiting model \eqref{eq:p:lim}, that is
\begin{equation*}
	\left\{
	\begin{array}{l}
		\displaystyle \cL_{21} =	\int
		\sgn{\left(\bar{p}-\bar{q}^\eps\right)}(\bv)\Lambda\left(y-D_t N(\bv'),\bv,\bv'\right)\left(\bar{p}-\bar{q}^\eps\right)\left(\bv'\right)\cM(y)\,
		\dD \bx\dD \bv \dD \bv'\dD y\,
		,\\[1.5em]
		\displaystyle \cL_{22} = 	\int
		\sgn{\left(\bar{p}-\bar{q}^\eps\right)}(\bv)\Lambda\left(y-D_t N(\bv),\bv',\bv\right)\left(\bar{p}-\bar{q}^\eps\right)\left(\bv\right)\cM(y)\,
		\dD \bx\dD \bv \dD \bv'\dD y\,,
	\end{array}
	\right.
\end{equation*}
we prove that $\cL_{21}-\cL_{22}$ is non-positive. The last term $\cL_{31}-\cL_{32}$ corresponds to the error due to the difference between $q^\eps$ and $\bar{q}^\eps\cM$, that is
\begin{equation*}
	\left\{
	\begin{array}{l}
		\displaystyle \cL_{31} =	\int
		\sgn{\left(\bar{p}-\bar{q}^\eps\right)}(\bv)\Lambda\left(y-D_t N(\bv'),\bv,\bv'\right)\left(\bar{q}^\eps\cM-q^\eps\right)\left(\bv',y\right)
		\dD \bx\dD \bv \dD \bv'\dD y\,
		,\\[1.5em]
		\displaystyle \cL_{32} = 	\int
		\sgn{\left(\bar{p}-\bar{q}^\eps\right)}(\bv)\Lambda\left(y-D_t N(\bv),\bv',\bv\right)\left(\bar{q}^\eps\cM-q^\eps\right)\left(\bv,y\right)
		\dD \bx\dD \bv \dD \bv'\dD y\,,
	\end{array}
	\right.
\end{equation*}
 we estimate $\cL_{31}-\cL_{32}$ thanks to the convergence of $q^\eps$ toward $\bar{q}^\eps\cM$ ensured by Proposition \ref{energy:estimate}.\\

Taking the absolute value inside the integral in $\cL_{21}$ we obtain
\begin{align*}
\cL_{21}-\cL_{22}\leq
\int
\Lambda\left(y-D_t N(\bv'),\bv,\bv'\right)&\left|\bar{p}-\bar{q}^\eps\right|\cM(\bv',y)\\[0.8em]
&-
\Lambda\left(y-D_t N(\bv),\bv',\bv\right)\left|\bar{p}-\bar{q}^\eps\right|\cM(\bv,y)\,
\dD \bx\dD \bv \dD \bv'\dD y\,,
\end{align*}
and then notice that the right-hand side cancels inverting $\bv$ and $\bv'$ in the first term, that is
\[
\cL_{21}\,-\,\cL_{22}\,\leq\,0\,.
\]

To estimate $\cL_{31}-\cL_{32}$, we invert $\bv$ and $\bv'$ inside $\cL_{31}$ and take the absolute value inside both integrals, this yields
\[
\cL_{31}-\cL_{32}
\,\leq\,
2\int
\Lambda\left(y-D_t N(\bv),\bv',\bv\right)\left|\bar{q}^\eps\cM-q^\eps\right|\left(\bv,y\right)
\dD \bx\dD \bv \dD \bv'\dD y\,.
\]
Thanks to assumption \eqref{hyp:Lambda+} we estimate $\int
\Lambda\left(y-D_t N(t,\bx,\bv),\bv',\bv\right)\dD \bv'$ by its supremum over all $(\bx,\bv,y)\in \R^d\times V\times \R$, this yields
\[
\cL_{31}-\cL_{32}
\,\leq\,
C\left\|\,
q^\eps\,-\,\ols{q}^\eps\,
\mathcal{M}\,
\right\|_{L^1
	\left(
	\R^d\times V\times \R
	\right)
}\,.
\]

We turn to the last two terms $\cL_{11}$ and $\cL_{12}$. We perform the change of variable $y\leftarrow y - D_t M(\bv) + D_t N(\bv)$ in the first term composing $\cL_{12}$, this yields
\[
\cL_{12} =	\int
\sgn{\left(\bar{p}-\bar{q}^\eps\right)}(\bv)\Lambda\left(y-D_t N(\bv),\bv',\bv\right)\bar{p}\left(\bv\right)
\left(\cM\left(y+ D_t M(\bv)- D_t N(\bv) \right)-\cM(y)\right)
\dD \bx\dD \bv \dD \bv'\dD y\,.
\]
Then, we take the absolute value inside the integral, and thanks to assumption \eqref{hyp:Lambda+}, we estimate $\int
\Lambda\left(y-D_t N(t,\bx,\bv),\bv',\bv\right)\dD \bv'$ by its supremum over all $(\bx,\bv,y)\in \R^d\times V\times \R$, which yields
\[
\left|\cL_{12}\right| \leq	C \int
\bar{p}\left(\bv\right)
\int \left|\cM(y+ D_t M(\bv)- D_t N(\bv) )-\cM(y)\right|\dD y\,
\dD \bx\dD \bv\,.
\]
We obtain the same estimate for $\cL_{11}$ following the same computations and therefore, we deduce
\[
\cL_{11}-\cL_{12}\, \leq\,	C \int
\bar{p}\left(\bv\right)
\int \left|\cM\left(y+ D_t M(\bv)- D_t N(\bv) \right)-\cM(y)\right|\dD y\,
\dD \bx\dD \bv\,.
\]
To estimate the integral with respect to $y$ in the preceding right-hand side, we notice, on the one hand, that since $\cM$ is a probability measure, it holds 
\[
\int \left|\cM\left(y+ D_t M(\bv)- D_t N(\bv) \right)-\cM(y)\right|\dD y\,\leq\,2\,,
\]
for all $(t,\bx,\bv)\in\R^+\times \R^d \times V$. On the other hand, thanks to Csizár-Kullback, we also have
\begin{align*}
\int \left|\cM\left(y+ D_t M(\bv)- D_t N(\bv) \right)-\cM(y)\right|&\dD y\,\leq\\[-0.3em]
&\sqrt{2}
\left(\int \cM(y)\,\log{\left(\frac{\cM(y)}{\cM\left(y+ D_t M(\bv)- D_t N(\bv) \right)}\right)}\dD y\right)^{\frac{1}{2}}.
\end{align*}
Hence, after simple computations involving $\cM$, we obtain
\begin{equation*}
\cL_{11}-\cL_{12}\, \leq\,	C \int
\bar{p}\left(\bv\right)\min{\left(1\,,\left|D_t M(\bv)- D_t N(\bv) \right|\right)}
\dD \bx\dD \bv\,.
\end{equation*}
We estimate $D_t M - D_t N$ thanks to \eqref{estimate:N:2} in Lemma \ref{estimate:N} and apply \eqref{moments:v:p} to bound moments of $\bar{p}$ with respect to $\bv$. This yields
\begin{equation*}
	\cL_{11}-\cL_{12}\, \leq\,	C \min{\left(1\,,\left(\eps+e^{-\frac{t}{\eps}}\right)t\right)}\,.
\end{equation*}

Gathering our estimates, we deduce 
\[
\frac{\dD}{\dD t}
\left\|\bar{p}-\bar{q}^\eps\right\|_{L^1\left(\R^d\times V\right)}
\,\leq\,C\left(\left\|\,
q^\eps\,-\,\ols{q}^\eps\,
\mathcal{M}\,
\right\|_{L^1
	\left(
	\R^d\times V\times \R
	\right)
}
+
\min{\left(1\,,\left(\eps+e^{-\frac{t}{\eps}}\right)t\right)}
\right)
\,.
\]
Then, we integrate the preceding relation between $0$ and $t$ and obtain
\[
\left\|\bar{p}(t)-\bar{q}^\eps(t)\right\|_{L^1}
\,\leq\,
\left\|\bar{p}_0-\bar{q}^\eps_0\right\|_{L^1}
+\,
C\left(\int_0^t \left\|\,
q^\eps(s)\,-\,\ols{q}^\eps(s)\,
\mathcal{M}\,
\right\|_{L^1
}\dD s
+
\min{\left(t\,,\eps\,t^2+ \eps^2\right)}
\right)
\,.
\]
Applying Jensen inequality and then Proposition \ref{energy:estimate} to bound the integral in the right-hand side, we obtain
\[
\left\|\bar{p}(t)-\bar{q}^\eps(t)\right\|_{L^1}
\,\leq\,
\left\|\bar{p}_0-\bar{q}^\eps_0\right\|_{L^1}
+\,
C\left(\sqrt{\eps\, t}\,E[q^\eps_0]^{\frac{1}{2}} \,+\,
\sqrt{\eps\, t}\,\left(t \,+\, 1\right)
+
\min{\left(t\,,\eps\,t^2+ \eps^2\right)}
\right)
\,.
\]
Then, we notice that: $\min{\left(t\,,\eps\,t^2+ \eps^2\right)}\leq t\min{\left(1\,,\eps\,t\right)} + \eps^2$, and that: $\min{\left(1\,,\eps\,t\right)} \leq \sqrt{\eps}\,\left(t \,+\, 1\right)/\sqrt{t}$, where to derive the second inequality, we distinguish the cases $t>1/\eps$ and $t\leq1/\eps$ separately. Hence, we deduce the expected result
\[
\left\|\bar{p}(t)-\bar{q}^\eps(t)\right\|_{L^1}
\,\leq\,
\left\|\bar{p}_0-\bar{q}^\eps_0\right\|_{L^1}
+\,
C\left(\sqrt{\eps\, t}\,E[q^\eps_0]^{\frac{1}{2}} \,+\,
\sqrt{\eps\, t}\,\left(t \,+\, 1\right)
+
\eps^2\right)
\,.
\]
\end{proof}
\section{Proof of Theorem \ref{TH:MAIN}}
To prove Theorem \ref{TH:MAIN}, we invert the change of variable \eqref{change:var} in the estimates of Propositions \ref{energy:estimate} and \ref{cv:q:p:bar} in order to translate the results from the re-scaled distribution $q^\eps$ to the initial quantity $p^\eps$. 

Let us start with the first estimate in Theorem \ref{TH:MAIN}. According to Proposition \ref{energy:estimate}, it holds 
\[
\int_0^t \left\|\,
q^\eps(s)\,-\,\ols{q}^\eps(s)\,
\mathcal{M}\,
\right\|^2_{L^1
	\left(
	\R^d\times V\times \R
	\right)
}\dD s\,\leq\,\eps\,\int_{\R^d\times V\times \R}
q^\eps_0
\log{\left(\frac{q^\eps_0}{\cM}\right)}(\bx,\bv,y)\dD \bx\dD \bv \dD y \,+\,
C\,\eps\left(t^2 \,+\, 1\right)\,,
\]
for all times $t\geq0$. Inverting \eqref{change:var} in both left and right-hand sides, the estimate becomes
\[
\int_0^t \left\|\,
p^\eps(s)\,-\,\ols{p}^\eps\,
\mathcal{M}_{\eps,N}(s)\,
\right\|^2_{L^1\left(
	\R^d\times V\times \R
	\right)
}\dD s\,\leq\,\eps\,\int
p^\eps_0
\log{\left(\frac{p^\eps_0}{\cM_{\eps,M_0}}\right)}(\bx,\bv,m)\dD \bx\dD \bv \dD m\,+\,
C\,\eps\left(t^2 \,+\, 1\right)\,,
\]
where $\mathcal{M}_{\eps,N}$ is defined by \eqref{def:M:eps}. To estimate the relative energy on the right-hand side we decompose it as follows
\begin{align*}
\int
p^\eps_0
\log{\left(\frac{p^\eps_0}{\cM_{\eps,M_0}}\right)}\dD \bx\dD \bv \dD m\,=&\,
\int
p^\eps_0
\log{\left(p^\eps_0\right)}\dD \bx\dD \bv \dD m\\[0.8em]
\,+\,
\frac{1}{2}
&\int
\left|\frac{m-M_0(\bx)}{\eps}\right|^2
p^\eps_0
\dD \bx\dD \bv \dD m
-\log{\left(2\pi\right)}+\left|\log{\left(\eps\right)}\right|
\,.
\end{align*}
Both first and second terms on the right-hand side stay uniformly bounded as $\eps$ vanishes thanks to assumptions \eqref{hyp1:p:eps} and \eqref{hyp2:p:eps} respectively. Hence, we deduce 
\[
\int_0^t \left\|\,
p^\eps(s)\,-\,\ols{p}^\eps\,
\mathcal{M}_{\eps,N}(s)\,
\right\|^2_{L^1\left(
	\R^d\times V\times \R
	\right)
}\dD s\,\leq\,
C\,\eps\left(t^2 \,+\, 1\,+\,\left|\log{\left(\eps\right)}\right|\right)\,,\quad\forall\,t \in\R^+\,.
\]
Then, the first estimate in Theorem \ref{TH:MAIN} is obtained taking the square root on both sides of the preceding inequality.

To obtain the second estimate in Theorem \ref{TH:MAIN}, we invert \eqref{change:var} in the estimate of Proposition \ref{cv:q:p:bar}. Since computations are very similar to the ones presented above, we do not detail them. To conclude this proof, the last estimate in Theorem  \ref{TH:MAIN}, which ensures the convergence of $N$ towards $M$, is obtained thanks to \eqref{estimate:N:2} in Lemma \ref{estimate:N:1}.
\section{Conclusion and perspectives}
\label{sec:5}
In this article, we investigated a possible description of the internal mechanisms allowing bacteria to feel gradients. The main point of our analysis is that we obtain strong convergence with explicit rates towards the phenomenological run-and-tumble equation, departing from a model incorporating internal variables. In order to capture the complexity of real-life applications, we proposed a robust method, which allows for general tumbling kernel. Numerous perspectives arise from this work.\\

A natural improvement of our result consists in removing the well-preparedness condition \eqref{hyp2:p:eps}. The same difficulty was successfully dealt with in the context of concentration phenomena for neural networks \cite{BF,Blaustein23strFHN,blaustein_bouin22}. Since \eqref{eq:p} and the model studied in \cite{Blaustein23strFHN} share a similar structure, the arguments developed in \cite{Blaustein23strFHN} might adapt in our context.\\

Another interesting perspective would be to extend our analysis to more general internal dynamics \cite{Erban_Othmer04,Erban_Othmer05}. Nonlinear dynamics should be in the range of our method but in this case, the limiting tumbling kernel $\bar{\Lambda}$ would be given as the convolution product between $\Lambda$ and the stationary solution to a Fokker-Planck equation with internal dynamics dependent coefficients. This would yield interesting insights on the interplay between $\Lambda$, the internal dynamics and the resulting tumbling kernel $\bar{\Lambda}$. It could complement former detailed studies on that matter \cite{Xue15}. Including fast variables could lead to very interesting and biologically relevant questions as it is not clear how scale fast variables compared to $\eps$. Biological experiments  \cite{Sourjik_Berg02,Sourjik04,Jiang_Ouyang_Tu10} provide insights on the multiple time scales involved in the signaling pathways and should therefore constitute the starting point to a further investigation.\\

In line with the last paragraph, it would be of great interest to investigate the scaling in $\eps$ of the diffusion term $\partial_m^2 p^\eps$ in \eqref{eq:p}. This term represents random fluctuations which occur at different levels in the signaling pathways of bacteria \cite{Clausznitzer_Endres11}. Some detailed models exist \cite{Clausznitzer_Endres11} but it does not seem straightforward how the noise intensity scales as $\eps\rightarrow 0$. In the mathematical literature, neglecting random fluctuations completely has been considered in \cite{Perthame_Tang_vauchelet16}. Obtaining strong and quantitative convergence estimates in this case would constitute a challenge. Intermediate situations, where the diffusion scales as $\eps^\alpha \partial_m^2 p^\eps$ with $\alpha>1$, might also yield interesting corrections to the resulting tumbling kernel $\bar{\Lambda}$. Indeed, the formal computations presented in the introduction lead to a corrected $\bar{\Lambda}^\eps$ given by 
\[
\bar{\Lambda}^\eps(m,\bv,\bv')\,=\,
\int_{\R}
\Lambda\left(y-m,\bv,\bv'\right)\frac{1}{\eps^{\alpha-1}}\, \mathcal{M}\left(\frac{y}{\eps^{\alpha-1}}\right)\dD y\,,
\]
where the convolution with the Gaussian distribution vanishes in the limit $\eps\rightarrow 0$, which is consistent with the findings of \cite{Perthame_Tang_vauchelet16}.\\

To conclude, coupling equation \eqref{eq:p} with an evolution equation for $M$ is of primary interest on both the biological and mathematical point of view. On the mathematical point of view, it would lead to an intricate analysis as \eqref{eq:p} would become nonlinear. On the biological point of view, this model would hopefully capture complex patterns observed experimentally \cite{Budrene_Berg91}. Quantifying rigorously these patterns would lead to very interesting developments.

%
\section*{Acknowledgement}
I thank Benoît Perthame for introducing me to the run and tumble model and for enlightening discussions and advises over the content of this article. \\

I also thank Francis Filbet for his precious advises and comments along the elaboration of the manuscript.

\bibliographystyle{abbrv}
\bibliography{refer}
\end{document}